\documentclass{amsart}

\usepackage{amsfonts,amssymb,verbatim,amsmath,amsthm,latexsym,textcomp,amscd}
\usepackage{latexsym,amsfonts,amssymb,epsfig,verbatim}
\usepackage{amsmath,amsthm,amssymb,latexsym,graphics,textcomp}
\usepackage{graphicx}

\usepackage{url}  
\usepackage{psfrag}
\usepackage{amsmath,amsthm,amsfonts,amssymb,mathrsfs,bm,graphicx,stmaryrd}
\usepackage{mathtools}
\usepackage[usenames]{color}
\usepackage{hyperref}
\usepackage [latin1]{inputenc}

\usepackage{tikz}
\usetikzlibrary{shapes,arrows}

\begin{document}

\tikzstyle{decision} = [diamond, draw, fill=gray!20, 
    text width=4.5em, text badly centered, node distance=3cm, inner sep=0pt]
\tikzstyle{block} = [rectangle, draw, fill=gray!20, 
    text width=10em, text centered, rounded corners, minimum height=3em]
\tikzstyle{line} = [draw, -latex']
\tikzstyle{cloud} = [draw, ellipse,fill=gray!20, node distance=3cm,
    minimum height=2em]

\newtheorem{theorem}{Theorem}[section]
\newtheorem{prop}[theorem]{Proposition}
\newtheorem{assume}[theorem]{Assumption}
\newtheorem{lemma}[theorem]{Lemma}
\newtheorem{cor}[theorem]{Corollary}
\newtheorem{definition}[theorem]{Definition}
\newtheorem{conj}[theorem]{Conjecture}
\newtheorem{claim}[theorem]{Claim}
\newtheorem{qn}[theorem]{Question}
\newtheorem{defn}[theorem]{Definition}
\newtheorem{defth}[theorem]{Definition-Theorem}
\newtheorem{obs}[theorem]{Observation}
\newtheorem{rmk}[theorem]{Remark}
\newtheorem{ans}[theorem]{Answers}
\newtheorem{slogan}[theorem]{Slogan}
\newtheorem{corollary}[theorem]{Corollary}
\newtheorem{proposition}[theorem]{Proposition}
\newtheorem{observation}[theorem]{Observation}
\newtheorem{property}{Property}[subsection]
\newtheorem{remark}[theorem]{Remark}
\newtheorem{example}[theorem]{Example}

\newtheorem{question}{Question}
\newtheorem{maintheorem}{Theorem}
\newtheorem{maincoro}[maintheorem]{Corollary}
\newtheorem{mainprop}[maintheorem]{Proposition}

\newcommand{\bluecomment}[1]{\textcolor{blue}{#1}}
\newcommand{\boundary}{\partial}
\newcommand{\hhat}{\widehat}
\newcommand{\C}{{\mathbb C}}
\newcommand{\B}{{\mathbb B}}
\newcommand{\Ga}{{\Gamma}}
\newcommand{\G}{{\Gamma}}
\newcommand{\s}{{\Sigma}}
\newcommand{\PSL}{{PSL_2 (\mathbb{C})}}
\newcommand{\pslc}{{PSL_2 (\mathbb{C})}}
\newcommand{\pslr}{{PSL_2 (\mathbb{R})}}
\newcommand{\Gr}{{\mathcal G}}
\newcommand{\integers}{{\mathbb Z}}
\newcommand{\natls}{{\mathbb N}}
\newcommand{\ratls}{{\mathbb Q}}
\newcommand{\reals}{{\mathbb R}}
\newcommand{\proj}{{\mathbb P}}
\newcommand{\lhp}{{\mathbb L}}
\newcommand{\tube}{{\mathbb T}}
\newcommand{\cusp}{{\mathbb P}}
\newcommand\AAA{{\mathcal A}}
\newcommand\HHH{{\mathbb H}}
\newcommand\BB{{\mathcal B}}
\newcommand\CC{{\mathcal C}}
\newcommand\DD{{\mathcal D}}
\newcommand\EE{{\mathcal E}}
\newcommand\FF{{\mathcal F}}
\newcommand\GG{{\mathcal G}}
\newcommand\HH{{\mathcal H}}
\newcommand\II{{\mathcal I}}
\newcommand\JJ{{\mathcal J}}
\newcommand\KK{{\mathcal K}}
\newcommand\LL{{\mathcal L}}
\newcommand\MM{{\mathcal M}}
\newcommand\NN{{\mathcal N}}
\newcommand\OO{{\mathcal O}}
\newcommand\PP{{\mathcal P}}
\newcommand\QQ{{\mathcal Q}}
\newcommand\RR{{\mathcal R}}
\newcommand\SSS{{\mathcal S}}
\newcommand\TT{{\mathcal T}}
\newcommand\UU{{\mathcal U}}
\newcommand\VV{{\mathcal V}}
\newcommand\WW{{\mathcal W}}
\newcommand\XX{{\mathcal X}}
\newcommand\YY{{\mathcal Y}}
\newcommand\ZZ{{\mathcal Z}}
\newcommand{\iid}{{i.i.d.\ }}
	\renewcommand{\ae}{{a.e.\ }}
\newcommand\CH{{\CC\Hyp}}
\newcommand{\Chat}{{\hat {\mathbb C}}}
\newcommand\MF{{\MM\FF}}
\newcommand\PMF{{\PP\kern-2pt\MM\FF}}
\newcommand\ML{{\MM\LL}}
\newcommand\PML{{\PP\kern-2pt\MM\LL}}
\newcommand\GL{{\GG\LL}}
\newcommand\Pol{{\mathcal P}}
\newcommand\half{{\textstyle{\frac12}}}
\newcommand\Half{{\frac12}}
\newcommand\Mod{\operatorname{Mod}}
\newcommand\Area{\operatorname{Area}}
\newcommand\ep{\epsilon}
\newcommand\Hypat{\widehat}
\newcommand\Proj{{\mathbf P}}
\newcommand\U{{\mathbf U}}
 \newcommand\Hyp{{\mathbf H}}
\newcommand\D{{\mathbf D}}
\newcommand\Z{{\mathbb Z}}
\newcommand\R{{\mathbb R}}
\newcommand\Q{{\mathbb Q}}
\newcommand\E{{\mathbb E}}
\newcommand\EXH{{ \EE (X, \HH_X )}}
\newcommand\EYH{{ \EE (Y, \HH_Y )}}
\newcommand\GXH{{ \GG (X, \HH_X )}}
\newcommand\GYH{{ \GG (Y, \HH_Y )}}
\newcommand\ATF{{ \AAA \TT \FF }}
\newcommand\PEX{{\PP\EE  (X, \HH , \GG , \LL )}}
\newcommand{\lct}{\Lambda_{CT}}
\newcommand{\lel}{\Lambda_{EL}}
\newcommand{\lgel}{\Lambda_{GEL}}
\newcommand{\lre}{\Lambda_{\mathbb{R}}}

\newcommand\til{\widetilde}
\newcommand\length{\operatorname{length}}
\newcommand\tr{\operatorname{tr}}
\newcommand\cone{\operatorname{cone}}
\newcommand\gesim{\succ}
\newcommand\lesim{\prec}
\newcommand\simle{\lesim}
\newcommand\simge{\gesim}
\newcommand{\simmult}{\asymp}
\newcommand{\simadd}{\mathrel{\overset{\text{\tiny $+$}}{\sim}}}
\newcommand{\ssm}{\setminus}
\newcommand{\diam}{\operatorname{diam}}
\newcommand{\pair}[1]{\langle #1\rangle}
\newcommand{\T}{{\mathbf T}}
\newcommand{\I}{{\mathbf I}}
\newcommand{\pG}{{\partial G}}
 \newcommand{\oxi}{{[1,\xi)}}
\newcommand{\cg}{\mathcal{G}}

\newcommand{\tw}{\operatorname{tw}}
\newcommand{\base}{\operatorname{base}}
\newcommand{\trans}{\operatorname{trans}}
\newcommand{\rest}{|_}
\newcommand{\bbar}{\overline}
\newcommand{\lbar}{\underline}
\newcommand{\UML}{\operatorname{\UU\MM\LL}}
\newcommand{\EL}{\mathcal{EL}}
\newcommand{\ncox}{{N_C([o,\xi))}}
\newcommand{\qle}{\lesssim}

\def\ind{{\mathbf 1}}
\def\N{\mathbb{N}}
\def\L{\mathbb{L}}
\def\P{\mathbb{P}}
\def\Z{\mathbb{Z}}
\def\R{\mathbb{R}}
\def\S{\mathcal{S}}
\def\X{\mathcal{X}}
\def\E{\mathbb{E}}
\def\l{\ell}
\def\Ups{\Upsilon}
\def\om{\omega}
\def\Om{\Omega}

\newcommand\Gomega{\Omega_\Gamma}
\newcommand\nomega{\omega_\nu}
\newcommand\omegap{{(\Omega,\P)}}
\newcommand\omegapp{{(\Omega',\P)}}

\DeclarePairedDelimiter\floor{\lfloor}{\rfloor}
\newcommand{\cF}{\mathcal{F}}
\newcommand{\cD}{\mathcal{D}}
\newcommand{\cZ}{\mathcal{Z}}
\newcommand{\cf}{\mathcal{F}}
\newcommand{\cA}{\mathcal{A}}
\newcommand{\cB}{\mathcal{B}}
\newcommand{\cC}{\mathcal{C}}
\newcommand{\cT}{\mathcal{T}}
\newcommand{\rb}{\mathfrak{A}}
\newcommand{\Var}{\hbox{Var}}
\newcommand{\ee}{\hbox{e}_1}
\newcommand{\cd}{\mathcal{D}} 
\newcommand{\ce}{\mathcal{E}}
\renewcommand{\ind}{\mathrm{ind}\,}
\setcounter{tocdepth}{2}

\title[Exceptional directions in hyperbolic FPP]{Geodesic Trees and Exceptional Directions in FPP on Hyperbolic Groups}

\author{Riddhipratim Basu}
\address{Riddhipratim Basu, International Centre for Theoretical Sciences, Tata Institute of Fundamental Research, Bengaluru, India}

\email{rbasu@icts.res.in}

\author{Mahan Mj}
\address{Mahan Mj, School
of Mathematics, Tata Institute of Fundamental Research. 1, Homi Bhabha Road, Mumbai-400005, India}

\email{mahan@math.tifr.res.in}

\subjclass[2010]{60K35, 82B43,  20F67 (20F65, 51F99, 60J50) } 

\keywords{first passage percolation, hyperbolic group,  coalescence, exceptional direction, Cannon-Thurston map}

\date{\today}

 \begin{abstract}
We continue the study of the geometry of infinite geodesics in first passage percolation (FPP) on Gromov-hyperbolic  groups $G$, initiated in \cite{benjamini-tessera} and developed further in \cite{BM19}. It was shown in \cite{BM19} that,  given any fixed direction $\xi\in \partial G$, and under mild conditions on the passage time distribution, there exists almost surely a unique semi-infinite FPP geodesic  from each $v\in G$ to $\xi$. Also, these geodesics coalesce to form a tree.  Our main topic of study is the set of (random) \emph{exceptional directions} for which uniqueness or coalescence fails.  We study these directions in the context of two random geodesics trees: one formed by the union of all geodesics starting at a given base point, and the other formed by the union of all semi-infinite geodesics in a given direction $\xi\in \partial G$. We show that, under mild conditions, the set of exceptional directions almost surely has a strictly smaller Hausdorff dimension than the boundary, and hence has measure zero with respect to the Patterson-Sullivan measure.  We also establish an upper bound on the maximum number of disjoint geodesics in the same direction. For groups that are not virtually free, we show that almost surely exceptional directions exist and are dense in $\pG$. When the topological dimension of $\pG$ is greater than one, we establish the existence of uncountably many exceptional directions. When the topological dimension of $\pG$ is $n$, we prove the existence of directions $\xi$ with at least $(n+1)$ disjoint geodesics. Our results hinge on deep facts about hyperbolic groups, in particular on establishing the existence of random Cannon-Thurston maps for the geodesic trees. En route, we also establish facts about the structure of random bigeodesics that substantially strengthen prior results from  \cite{benjamini-tessera}.  
\end{abstract}

\maketitle

\tableofcontents

\section{Introduction}\label{sec-intro} 
First passage percolation (FPP) is a well-known model of a random metric on a graph where edge is endowed with an independent random length coming from a common distribution taking non-negative values. Introduced by Hammersley and Welsh in 1965 \cite{HW} in the context of the nearest neighbor graph of Euclidean lattices, it has been extensively studied for the last six decades in the probability literature. Interest in this topic is also fueled by its connection to statistical physics; FPP on Euclidean lattices is expected to belong to the so-called Kardar-Parisi-Zhang (KPZ) universality class of random growth models. The study of FPP in Euclidean spaces is  notoriously hard; but there are a few related \emph{integrable} models in the planar setting for which the physics predictions have been verified. The question of universality remains out of reach. We refer the interested reader to \cite{ADH} for an excellent monograph on the history of the developments on FPP and \cite{Cor12} for an overview of related integrable models. 

Over the last decade or two, there has been growing interest in understanding the interplay between the randomness of the metric and the underlying background geometry.  FPP has been studied in different background geometries, typically on Cayley graphs of finitely generated groups (see e.g.\ \cite{BT15,AG23,G24}). The study of FPP on Gromov hyperbolic groups was initiated in \cite{bz-tightness,benjamini-tessera} and was later continued in \cite{BM19}. It was shown that while in some respects FPP on hyperbolic groups exhibit similar behavior  to what is known or expected in the Euclidean (more specifically planar Euclidean) case, in other respects it shows sharply contrasting behavior. {We also point out the recent papers \cite{JQ25,BJQ25} where FPP is studied in more general spaces having hyperbolic properties along certain directions called Morse geodesics.}

In this paper, we continue the investigation from \cite{BM19}. We focus on the study of the structure of the geodesics. {We answer a number of questions, analogs of which have been studied in the discrete Euclidean  case, but so far  been answered only in some planar exactly solvable models or some continuum models with additional rotational symmetries. As far as we are aware, our results are the first ones for FPP in a hyperbolic background geometry and are also the first ones in  discrete models that are not exactly solvable. Our results also exhibit some qualitatively new features in the higher dimensional setting.} 

In this paper, our main object of study is the structure of random FPP geodesics on a Cayley graph  $\Gamma$ of a non-elementary Gromov hyperbolic group $G$. The geometry of geodesics is a very well studied topic in the Euclidean set-up and is of fundamental importance (we shall say more about this later). Under mild assumptions on the passage time distribution, almost surely there is a unique random geodesic between any two vertices (i.e., there are no geodesic bigons). The existence of semi-infinite geodesic rays can also be established easily. The collection of  geodesics thus naturally display a tree-like structure. Our main objects of study in this paper are the following two types of random geodesic trees. Consider the union of all geodesics starting at a given $v\in \Gamma$. By uniqueness of geodesics this is almost surely a tree, and we refer to this as the  \emph{forward random geodesic tree} rooted at $v$. On the other hand, for a fixed $\xi$ in $\pG$, the boundary of $G$, it was shown in \cite{BM19} that for any $v\in \Gamma$,  there exists  almost surely a unique semi-infinite geodesic from $v$ in the direction $\xi$. Further,
 these geodesics (indexed by $v\in \Gamma$) \emph{coalesce}. Hence the union of all these semi-infinite geodesics in a fixed direction $\xi$ form a tree. We call it the \emph{backward random geodesic tree} in the direction $\xi$. For brevity, we shall usually refer to these as \emph{forward tree} and \emph{backward tree} respectively. 

These trees are interesting objects of study in the probability literature. There is a geometric motivation  as well. In the context of geodesic flow on the unit tangent bundle $U\til{M}$ of a complete simply connected manifold $\til M$ of pinched negative curvature, points $\xi \in \partial \til{M}$ parametrize the widely studied \emph{stable manifolds} 
$\SSS_\xi$ in  $U\til{M}$. (We typically think of $\til M$ as the universal cover of a closed manifold $M$.) Here, $\SSS_\xi \subset U\til{M}$ consists of geodesics $\gamma \subset  U\til{M}$ such that $\gamma (\infty) = \xi$.
Thus, the backward random tree in the direction $\xi$ is the random discrete analog of stable manifolds
$\SSS_\xi$.  

 For both forward and backward geodesic trees, of particular  interest is the behavior of exceptional random directions, i.e.\  random directions which show atypical behavior. For example, for the forward tree, recall that in any given direction $\xi$, there is almost surely a unique semi-infinite geodesic. However, since the number of directions is uncountable this allows for the existence of random directions where there might be multiple disjoint semi-infinite geodesics. Such directions are referred to as \emph{exceptional directions} (a similar notion works for backward trees as well, precise definitions will be given later).
 
{A standard theme in probability theory and random geometry is to first study typical events, and then to try and characterize atypical points where these typical events fail. Here, the typical event we are dealing with is coalescence of random geodesics. In \cite{BM19} we established that coalescence is typical behavior in hyperbolic FPP. Thus, atypical behavior is (almost tautologically) non-coalescence of geodesics. The study of exceptional directions carried out in this paper systematizes the study of this atypical behavior and fits into this general theme. We elaborate on this below in Section~\ref{sec-motvn}.}

There many natural questions about exceptional directions one can ask. Are there exceptional directions? How many are there? How many disjoint geodesics can there be in an exceptional direction? These questions have been investigated in the Euclidean setting but the complete answer has so far been obtained only in the set-up of planar exponential last passage percolation (LPP) model and its scaling limit \emph{the directed landscape}. We shall briefly recall the state of the literature in Section \ref{s:lit} below. For hyperbolic FPP, under mild conditions on the edge length distribution, we provide an almost complete answer to these questions (see Theorem \ref{thm-omni-intro} for a summary of the statements). By way of our proof, we also obtain results of independent interest about the structure of bigeodesics and geodesic trees which significantly strengthens previous results obtained in \cite{benjamini-tessera, BM19}.

Our results are proved by a combination of probabilistic and geometric arguments, in particular several deep facts about hyperbolic groups and their boundaries are crucially used. A principal tool is a random Cannon-Thurston map for the natural embeddings of the forward tree and the backward tree into $\Gamma$ for almost every realization $\omega$ of the random edge lengths.  The existence of random Cannon-Thurston maps is derived from probabilistic estimates--a new conjunction of geometric and probabilistic methods. We have tried to make the paper accessible to a larger audience by recalling, at various stages,  the tools that go into it. A more detailed discussion on the arguments is provided in Section \ref{sec-org} below.

{
\subsection{Motivation behind exceptional directions}\label{sec-motvn}
Since this paper is intended for a mixed audience of probabilists and hyperbolic geometers, it would be appropriate to motivate the notion of 
\emph{exceptional directions} particularly for the latter.

{The questions investigated in this paper flow from a rather common motif in probability theory, particularly in the area of random geometry. Often, a particular geometric property is shown to hold almost surely for every deterministic point of an uncountable set, typically points of the underlying space or time points in an interval. However, since there are uncountably many such points, it might well be the case that with positive probability, and often with probability one, there are random locations or times which give rise to some atypical phenomena/structures.  These phenomena/structures almost surely do not arise at deterministic/particular locations or times. One then studies what atypical phenomena/structures can occur. A measure of the rarity of such 
	phenomena/structures is given by the Hausdorff dimension of the set of points where such atypical structures arise. 
	
	Understanding such atypical structures and their rarity has turned out to be useful in understanding the fractal structure of many random geometry models and has also given rise to new mathematics. One famous example in this vein is the case of infinite clusters in critical dynamical percolation on the plane where it has been shown that although at any fixed time there are no infinite clusters almost surely, there are random times when infinite clusters emerge \cite{HPS97,SS10, GPS10}. Similar questions have also been studied in dynamical discrete webs and Brownian webs \cite{FNRS09}. Atypical directions have been studied in certain models of random trees in the real hyperbolic plane in \cite{CFT23}. 
	
	In this paper we are concerned with models equipped with a  random metric. In many interesting random metric models, a typical phenomenon is \emph{coalescence of geodesics}, i.e., geodesics to a fixed point or a point at infinity starting from different points coalesce almost surely before reaching the endpoint. Since the number of points is uncountable, the following questions naturally arise: 
	\begin{itemize}
	\item Are there atypical points where this coalescence structure fails?
	\item What atypical coalescence structures are possible?
	\end{itemize} 
	 These questions have been extensively studied on planar continuum random geometric models of the Brownian map (see e.g.\ \cite{Le22}) and the directed landscape (see e.g.\ \cite{Dau23, Bha23b}). The latter example, the directed landscape, is the putative universal scaling limit of planar first and last passage percolation models in Euclidean background geometry. The study of exceptional directions (i.e., the set of random directions where the coalescence of semi-infinite geodesic rays fails) there is the major motivation for the present work. We shall give a more detailed account for the known results in these models in Section \ref{s:lit}. In \cite{BM19}, we showed that, in first passage percolation on hyperbolic groups, under mild conditions, coalescence of semi-infinite geodesics holds almost surely in any fixed direction of the Gromov boundary. In this work, we investigate the natural question of understanding the exceptional directions for hyperbolic FPP and explore similarities and differences with known or predicted results in the Euclidean case.} 

}

\subsection{Basic set-up and statement of main results}
\label{sec-prel} Before stating the main results, we fix the setup with which we shall be working for the rest of this paper. 
\begin{itemize}
\item  $G$--non-elementary hyperbolic group.
\item $\Ga=(G,S)$-- Cayley graph of $G$ with respect to a finite symmetric generating set $S$. The  size of $S$ (equaling the valence of any vertex of $\Gamma$) is $D$. Let $E$ denote the set of edges in $\Gamma$. $\Gamma$ is  $\delta$-hyperbolic  with respect to the word metric. The identity element of $G$ (as well as the corresponding vertex in $\Gamma$) will be denoted by $1$.
\item $\partial G$-- Gromov boundary of $G$, given by asymptote classes of geodesics in the Cayley graph.
\item $\nu$--Patterson-Sullivan measure on   $\partial G$, given by a weak limit of uniform measure on $n$-spheres in the Cayley graph as $n\to \infty$.
\item $\rho$--edge-length distribution, a measure on $\R_{+}$.
\item $\Omega$-- configuration space, given by $[0, \infty)^E$ equipped with  product measure $\P=\rho^{\otimes E}$.
\item $\omega$--element of $\omegap$, i.e., $\omega(e)$ are i.i.d. random variables distributed according to the measure $\rho$.
\item $d_{\omega}$--the first passage metric  on $\Gamma$ corresponding to $\omega\in \Omega$, i.e., 
$$ d_{\omega}(x,y)=\inf_{\gamma:x\to y} \sum_{e\in \gamma} \omega(e).$$
\item Geodesics-- $[x,y]$ shall denote a word geodesic between $x,y\in \Gamma$, and $[x,y]_{\omega}$ (called the $\omega$-geodesic between $x$ and $y$) shall denote the geodesic between $x,y$ in the $d_{\omega}$ metric. 
\end{itemize}

{Notice that we have chosen not to formally define several standard objects in hyperbolic geometry, e.g., the Gromov boundary and the Patterson-Sullivan measure. We refer the reader to \cite[Section 2.2]{BM19} where relevant definitions and basic properties  are are collected.}

Throughout the paper we shall work with the following assumption on the passage time distribution $\rho$. 

\begin{assume}\label{assume-subexp}
Assume that $\rho$ does not have any atoms and has exponential tails, i.e.\ for some $a > 0$, $\int_0^\infty e^{ax} d\rho(x) $ is finite. 
\end{assume}

It is not hard to see that the non-atomic nature of $\rho$ in Assumption \ref{assume-subexp} implies that for a.e. $\omega$ there is a unique geodesic $[x,y]_{\omega}$ between each pair $x,y\in \Gamma$. Next we recall the definitions of semi-infinite geodesics (also known as geodesic rays), bigeodesics, and their directions.  

\begin{defn}\label{def-fppray} 
	For  $\om \in \Omega$, a semi-infinite (resp.\ bi-infinite) path $\sigma_\omega$ is said to be an $\om-$geodesic ray (resp.\ a bi-infinite
		$\om-$geodesic, or simply a bigeodesic) if every finite subpath of   $\sigma_\omega$ is an $\om-$geodesic.
	
	A path $\sigma$ is said to  accumulate on $\xi \in \pG$ if there exist $v_n \in  \sigma$ such that $v_n \to \xi$ as $n \to \infty$. 
		An $\om-$geodesic ray $\sigma_\omega$ accumulating on $\xi\in \pG$  has direction $\xi$ if its only accumulation point  in $\pG$ is $\xi$.
\end{defn}

We next recall some relevant results from \cite{BM19}.

\begin{rmk}
	Although the hypothesis on the passage time distribution in \cite{BM19} was somewhat stronger than Assumption \ref{assume-subexp}, all the results we need from \cite{BM19} in this paper (in particular, Theorems~\ref{bmdirexists} and \ref{bmthmcoalesce} below) hold under the weaker hypothesis in Assumption \ref{assume-subexp}; see \cite[Remark 4.5]{BM19}.
\end{rmk}

The following theorem says that almost surely all $\omega$-geodesic rays have a direction and that
every $\xi\in \pG$ is the direction of an $\omega$-geodesic ray. 

\begin{theorem}\cite[Theorems 6.6, 6.7]{BM19}
	\label{bmdirexists}
	Consider FPP on $\Gamma$ under Assumption \ref{assume-subexp}. For $o\in \Gamma$, and for $\P$- a.e. $\omega \in \Omega$,  $\om$-geodesic rays starting at $o$ have a direction $\xi\in \pG$.

	Fix $\xi \in \partial G$ and $x_{n}\in \Ga$ such that $x_{n}\to \xi$. For $\P$-a.e.\ $\om \in \Omega$ the sequence of $\om-$geodesics $[o,x_n]_\om$ from $o$ to $x_n$ converges (up to extracting a subsequence) to an $\om-$geodesic ray $[o,\xi)_\om$ having direction $\xi$.
\end{theorem}


We also need to recall the following result about coalescence. For FPP on $\Ga$, semi-infinite geodesics in a fixed direction almost surely coalesce, i.e.\ the set of edges in the symmetric difference of the two geodesic rays outside a sufficiently large ball centered at the identity element is empty.

\begin{theorem}\cite[Lemma 7.8, Theorem 7.9]{BM19}\label{bmthmcoalesce} 
	Consider FPP on $\Gamma$ under Assumption \ref{assume-subexp}. Given $\xi\in \partial G$, there exists a full measure subset $\Omega_\xi \subset \Omega$ such that for all  $\omega\in \Omega_\xi$, and each $o\in \Ga$,  there exists a unique $\omega$-geodesic ray (denoted $[o,\xi)_{\omega}$) starting from $o$ in direction $\xi$.
	
Further, given any direction $\xi$,	for  all  $\omega\in \Omega_\xi$, and any $o_1, o_2 \in G$
the  $\om-$geodesic rays $[o_1,\xi)_\om$ and $[o_2,\xi)_\om$ a.s.\ coalesce.
\end{theorem}

\begin{rmk}\label{rmk-clarfncoalesce}
Note that the quantification in  Theorem~\ref{bmthmcoalesce} above says that the full measure subset $\Omega_\xi$ depends on the point $\xi$ but not on $o_1, o_2 \in G$.
\end{rmk}

Next, we formally define the forward tree and the backward tree. 

\begin{defn}\label{def-fwdtree}
Consider FPP on $\Gamma$ with passage times satisfying that $\rho$ has no atoms. For $\P$-a.e. $\omega\in \Omega$, the union of random geodesics $$F(1,\omega):=\bigcup_{g \in G} [1,g]_\omega $$ is a tree, which we refer to as the  \emph{forward random geodesic tree}, or simply the forward tree,  of the FPP on $\Gamma$ with base point $1$. Forward random tree with base point $g\in G$ is defined similarly.  
\end{defn}

Observe that each semi-infinite $\omega$-geodesic ray starting at $1$ is contained in $F(1,\omega)$. 
To define the backward tree, notice that as an immediate consequence of Theorem~\ref{bmthmcoalesce}, we have the following.

\begin{cor}\label{cor-btree} 
		Given $\xi\in \partial G$, there exists a full measure subset $\Omega_\xi \subset \Omega$ such that for all  $\omega\in \Omega_\xi$,
$T(\xi,\omega) = \bigcup_{g \in G} [g,\xi)_\om$ is a tree.
\end{cor}

\begin{proof}
Since the existence of more than one $\omega-$geodesic between $g_1', g_2'\in G$ is a 
probability zero event, we can assume, by passing to a full measure subset
 $\Omega'$ of $\Omega$ if necessary that  $[g_1', \xi)_\om \cup [g_2',\xi)_\om$ contains no loops. Theorem~\ref{bmthmcoalesce} now guarantees that $[g_1',\xi)_\om \cup [g_2',\xi)_\om$ is a tree, in fact a tripod. 

Taking a union over all $g \in G$ and repeating the same argument
for all (countably many) pairs $g_1', g_2'$, we observe that $T(\xi,\omega) = \bigcup_{g \in G} [g,\xi)_\om$ is a tree.
\end{proof}

\begin{defn}\label{def-btree}
We shall refer to $T(\xi,\omega)$ as the \emph{backward random geodesic tree} , or simply the backward tree, from $\xi$ for
$\om \in \Omega_\xi$. 
\end{defn}

We are now ready to state our main results. From this point onward, all our results will be under Assumption \ref{assume-subexp}, this will not be mentioned explicitly henceforth. 

\medskip
\noindent
\textbf{Ends of trees and bigeodesics.} Our first result strengthens Theorem \ref{bmdirexists} by showing that on a full measure subset, a (not necessarily unique) semi-infinite geodesic $[1,\xi)_{\omega}$ exists simultaneously for all $\xi\in \pG$. (See Theorem~\ref{prop-ftree}.)

\begin{theorem}
\label{thm-ftree-intro}
There exists $\Omega_0\subset \Omega$ with $\P(\Omega_0)=1$ such that for all $\omega\in \Omega_0$ there exists a semi-infinite $\omega$-geodesic $[1,\xi)_{\omega}$ starting from $1$ in direction $\xi$ for all $\xi\in \pG$. 
\end{theorem}

A forward tree $F(1,\omega)\subset \Gamma$  is said to have \emph{complete ends} if  for all $\xi\in \pG$ there exists a semi-infinite  geodesic
$[1,\xi)_F \subset F$ such that $[1,\xi)_F$ accumulates precisely on $\xi$. Thus, Theorem~\ref{thm-ftree-intro} says that almost surely, $F(1,\omega)$ has 
complete ends.
We state next the corresponding result for the backward trees (see Theorem~\ref{thm-btree}).

\begin{theorem}\label{thm-btree-intro} For all $\xi \in \partial G$, 
there exists a full measure subset $\Omega_\xi\subset \Omega$ such that for $\om \in \Omega_\xi$,  the backward tree
$T(\xi,\omega)$ has complete ends, i.e.\ for all $\xi'\neq \xi$, there exists $p \in [1,\xi)_\om$ and a semi-infinite geodesic $(\xi',p]_\om$ such that $(\xi',p]_\om \subset T(\xi,\omega)$.
\end{theorem}

{
	\begin{rmk}\label{rmk-btree}
		Corollary~\ref{cor-btree} and hence Definition~\ref{def-btree} uses 
		Theorem~\ref{bmthmcoalesce} to define the backward tree. The full measure set
		$\Omega_\xi$ thus depends on $\xi\in \partial G$. 
		
		This dependence is essential for backward trees and cannot be removed (at least for groups that are not virtually free). In fact, if 
		we could pass to a full measure set that is independent of $\xi$, then 
		exceptional directions in the sense of Definition~\ref{def-except} below would simply not exist, contradicting the theorems in Section~\ref{sec-direx} for backward trees.
\end{rmk}}

{
A key tool that goes into proving Theorem~\ref{thm-btree-intro} is the following result which is of independent interest (see Theorem~\ref{thm-bi}).

\begin{theorem}\label{thm-bi-intro} There exists a full measure subset $\Omega_2 \subset \Omega$ such that  for all $\om \in \Omega_2$,
	and all $\xi_1 \neq \xi_2 \in \partial G$,
	there exists a bi-infinite $\omega-$geodesic $(\xi_1$,$\xi_2)_\omega$.
\end{theorem}

The main theorem
of \cite{benjamini-tessera} proves the existence of some  bi-infinite $\omega-$geodesic joining some $\xi_1 \neq \xi_2 \in \pG$ under mild assumptions on the edge distribution. In
the presence of  sub-exponentiality of  the edge distribution, 
 Theorem~\ref{thm-bi-intro} strengthens this conclusion  to the 
 existence of   bi-infinite $\omega-$geodesics joining \emph{any pair} $\xi_1 \neq \xi_2 \in \pG$.

\medskip
\noindent
\textbf{Exceptional directions.} First we formally define exceptional directions for forward and backward trees.  

 \begin{defn}\label{def-exceptionaldir} Let $G$ be hyperbolic.
 	For $\omega \in \Omega$, we say that $z \in \pG$ is an exceptional direction for the  forward random tree  $F(1,\omega)$ if there exist \emph{distinct} geodesics
 	$[1,z)^1_\omega$ and $[1,z)^2_\omega$ contained in 
 	 $F(1,\omega)$ such that $z\in \pG$ is the unique accumulation point for both $[1,z)^1_\omega$ and $[1,z)^2_\omega$.
 \end{defn}

\begin{rmk}
\label{r:altdef}
Notice that one could define exceptional directions without reference to the origin of the forward tree as simply the set of directions where coalescence fails, i.e., the set of all $\xi\in \pG$ such that (for a given $\omega\in \Omega$) there exist non-coalescing semi-infinite $\omega$-geodesics $[a,\xi)_{\omega}$ and $[b,\xi)_{\omega}$ with direction $\xi$. Although this is a priori a weaker definition compared to Definition \ref{def-exceptionaldir} above, Theorems~\ref{t:exceptional} and \ref{t:finite} (i.e.\ items (5), (6) of Theorem~\ref{thm-omni-intro} below) remain valid for this definition as well; see Remark \ref{r:altdef2}. Also, the lower bounds (items (1), (2), (4)) in Theorem~\ref{thm-omni-intro}) go through purely formally as this definition is weaker.
\end{rmk}

\begin{defn}\label{def-except}
For the backward random tree
$T(\xi,\omega)$, we say that 
$z\in \pG$ is an \emph{exceptional direction} if $T(\xi,\omega)$ contains \emph{disjoint} $\omega-$geodesic rays $[a,z)_\omega$ and 
$[b,z)_\omega$, both converging to $z \in \pG$. 
\end{defn}

Clearly, if $G$ is free, and $\Gamma$ is a standard Cayley graph, exceptional directions do not exist.

 \begin{defn}\label{def-mult}
 If $\xi'$ is an exceptional direction for either the  forward or the backward random tree, the cardinality of a maximal collection of disjoint $\omega-$geodesics converging to $\xi'$ is called the $\omega-$multiplicity of $\xi'$, or simply its
 \emph{multiplicity}, when $\omega \in \Omega$ is understood.
 \end{defn}
 
Thus, let $\{[a_i, \xi')_\omega: i \in I\}$ be a maximal collection of disjoint geodesic rays all converging to $\xi'$. Then the cardinality of $I$ is called the multiplicity of $\xi'$. Here, the collection is maximal in terms of its cardinality. Note that any 
 other such maximal collection  will necessarily have to be of the form 
$\{[b_i, \xi')_\omega: i \in I\}$, where $[a_i, \xi')_\omega$ and $[b_i, \xi')_\omega$ eventually agree for all $i \in I$.  
 
The basic questions on exceptional directions, as briefly alluded to before, are summarized below.
\begin{qn}\label{qn-exceptionaldirn} Given a hyperbolic $G$ with Cayley graph $\Gamma$, 
	\begin{enumerate}
	\item Do there exist exceptional directions for $F(1,\omega)$ and 
	$T(\xi,\omega)$ almost surely?
	\item If exceptional directions exist, what is the cardinality of the set of all exceptional directions?
	\item What is the maximum possible multiplicity of an exceptional direction?
	\item Are exceptional directions of measure zero with respect to the Patterson-Sullivan measure on $\pG$? What is the Hausdorff dimension of the set of exceptional directions? \\
	\end{enumerate}
\end{qn}

We are able to obtain rather complete answers to most of these questions for a hyperbolic groups $G$ (under Assumption \ref{assume-subexp}). For the convenience of reader, we summarize our main results on exceptional directions and their multiplicities in the following theorem. 

\begin{theorem}\label{thm-omni-intro} $ $
\begin{enumerate}
\item If $G$ is not free (up to finite index), almost surely exceptional directions in $\pG$ exist and are dense $\pG$ (Theorem~\ref{thm-nonfreedense}).
\item If $\pG$ has topological dimension $\dim_t \pG$ greater than 1, then almost surely there are 
\emph{uncountably many} exceptional directions (Theorem~\ref{thm-uncountable}).
\item If $G$ is a cocompact lattice in the hyperbolic plane, and the Cayley graph $\Gamma$ is planar, then  almost surely there are \emph{countably many} exceptional directions (Theorem~\ref{thm-countable}). 
\item Let $G$ be a  hyperbolic group such that $\dim_t \pG = n-1$. Then
there exists an exceptional direction $z \in \pG$ with multiplicity at \emph{at least $n$}. Thus there exist hyperbolic groups with arbitrarily large
multiplicity of an exceptional direction (in the sense of Question~\ref{qn-exceptionaldirn}(3)).
\item Exceptional directions have measure zero with respect to the Patterson-Sullivan measure (Theorem~\ref{t:exceptional}). In fact the set of exceptional directions has Hausdorff dimension (almost surely uniformly) smaller than that of $\pG$.
\item For any (fixed) $\Gamma$, there exists an almost sure uniform upper bound 
on the  multiplicity of an exceptional direction (in the sense of Question~\ref{qn-exceptionaldirn}(3)) (Theorem~\ref{t:finite}).
\end{enumerate} 
\end{theorem}
The statements above hold for both forward and backward random trees. 
Behind all but the last two results above, the main new ingredient is the existence of a \emph{random Cannon-Thurston map}, a continuous surjective map from the boundary of the tree to the boundary $\pG$ of the hyperbolic group $G$ (Theorem~\ref{thm-12proper}).

{For the forward tree, under further assumptions on the group $G$, we can strengthen Theorem \ref{thm-omni-intro}(2) further, and give a lower bound on the topological dimension of the set of exceptional directions; see Theorem \ref{thm-pd}.} 

\subsection{Known results in Euclidean FPP and LPP}
\label{s:lit}
{Some analogs of questions answered in Theorem \ref{thm-omni-intro} has been studied previously in various contexts, starting with FPP on Euclidean lattices, a notoriously hard model. As described below, there has been progress on this front, but sharper results have mostly been restricted to two related classes of models: planar exactly solvable models of last passage percolation, and certain models in the continuum (Euclidean $\R^2$ and $\R^m$) with additional rotational symmetries. In this section, we recall these results.}
%
%
These results, specifically the ones known in planar exactly solvable LPP models, believed to be in the same universality class as planar Euclidean FPP, served as a motivation from the Euclidean planar setup for us to address the analogous issues in a hyperbolic context. The actual results, however,  will not be used 
except to provide context in the rest of the paper.
 We start with describing the state of the art for FPP on Euclidean lattices.\\ 

\noindent
\textbf{FPP on Euclidean lattices.}
The interest in infinite geodesics in FPP go back forty years. In the 1980s, the question, attributed to Furstenberg, whether or not a bi-infinite geodesic exists in  Euclidean FPP was asked. This question remains open to date. The study of existence, uniqueness and coalescence of semi-infinite geodesics in given directions was taken up by Newman \cite{New95} and his collaborators since the mid nineties and has received extensive attention since then. However, this question is intimately related to the properties of the \emph{limit shape} for which we do not yet have an adequate understanding for FPP model on Euclidean lattices with generic edge length distributions. Therefore many of the existing results are conditional on various assumptions on the limit shape. In the planar case, some recent progress without such unverified assumptions have been made using the techniques of Busemann functions,  \cite{H05,hoffman, DH14, DH17,AH16}  being a representative sample of works in this direction. In particular, it is now known \cite{DH14,DH17,AH16} that, assuming differentiability of the boundary of limit shape, in every fixed direction uniqueness and coalescence of semi-infinite geodesics hold. In contrast to Theorem \ref{thm-btree-intro}, it was shown in  \cite[Theorem 1.10(4)]{DH14} that, under mild conditions, ``backward clusters" are almost surely finite. In \cite{AH16}, it was shown that there exists a family of \emph{random coalescing geodesics} corresponding to linear functionals that are supporting to the limit shape. It is also known that almost surely there do not exist any bigeodesics in a fixed direction. However, this does not rule out existence of bigeodesics in \emph{random directions}. Some recent results on coalescence under assumptions on the limit shape that are verifiable in certain cases were obtained in \cite{dembin-elboim-peled}. We refer the interested reader to \cite[Chapter 4]{ADH} for more details on the classically known results. Much of the progress in  recent years in planar FPP is contained in \cite{ahlberg2022numbergeodesicsplanarfirstpassage} 
and the survey paper \cite{ahlberg-survey}. Virtually nothing, however, is known in the higher dimensional setting. 

\medskip
\noindent
\textbf{Exactly solvable models and physical motivation.}
As mentioned above, the question of existence, cardinality and multiplicity of exceptional directions are poorly understood in lattice FPP, even in the planar setting. However, these questions are completely answered in the case of directed exponential LPP on $\Z^2$. In exponential LPP, one puts i.i.d.\ exponentially distributed weights on the vertices of $\Z^2$ and calculates the last passage time between two points by maximizing the sum of vertex weights over all up/right paths joining the two. Although last passage times form an anti-metric rather than a metric (and geodesics are defined as weight maximizing paths), it is believed that this model lies in the same KPZ universality class as planar FPP (under mild conditions). The following results are known in the case of exponential LPP (from universality considerations, one might believe the same results hold for planar FPP as well under some reasonable assumptions on the edge length distribution).

\begin{enumerate}
	\item Ferrari and Pimentel \cite{ferrari-pimentel} show that almost surely  all semi-infinite geodesics have a  direction. Also, almost surely for every direction in the first quadrant there is at least one semi-infinite geodesic in that direction. For every fixed non-axial direction, almost surely all semi-infinite geodesics in that direction coalesce. Consequently, exceptional directions are of measure 0.  
	\item Existence of exceptional directions with at least two distinct semi-infinite geodesics was shown by Ferrari, Martin,  and Pimentel \cite{ferrari-pimentel,fmp}. Coupier \cite{coupier11} showed that the set of exceptional directions is almost surely countable and dense in $[0,\pi/2]$.
	 
    \item  Coupier \cite{coupier11} (see also \cite{jrs}) also showed that almost surely there does not exist any non-axial direction with three non-coalescing see geodesics. 
  
\end{enumerate}

The last is often referred to as the N3G (no 3 geodesics) problem and has attracted a lot of attention in the literature. In \cite{jrs}, Janjigian,  Rassoul-Agha,  and Sepp\"{a}l\"{a}inen \cite{jrs} explored the geometry of geodesics and exceptional directions further and provided a comprehensive treatment, in particular they showed that there are two distinct geodesic rays from every point in every exceptional direction. See also the very recent \cite{ejs24} for other pathological behavior in a non-homogeneous (but still exactly solvable) exponential LPP model. Geodesics tree in a given direction (backward trees in our language) have also been studied for exponential LPP and precise tail estimates for sizes of sub-trees (backward clusters) have been obtained \cite{BBB23}. 

A richer structure has been established for the \emph{directed landscape}, which is the space time scaling limit of exponential LPP (and the putative universal scaling limit object in the KPZ class in 2 dimensions).
\begin{enumerate}
	\item Busani, Sepp\"{a}l\"{a}inen, and Sorensen \cite{bss} identify and describe a  countable dense set of directions  in which  at least two distinct geodesics exist from each initial point.
	\item  Busani \cite{busani2024nonexistencenoncoalescinginfinitegeodesics} shows that there exist no three non-coalescing  geodesics  in the same direction.
\end{enumerate}

{In the directed landscape,  exceptional directions correspond to the directions where multiple Busemann functions exist \cite{bss}. It was recently shown that they also correspond to the slopes for which the KPZ fixed point  (another universal scaling limit object in the KPZ class) has non-unique eternal solutions thereby violating the one-force-one-solution principle \cite{BBS25}.}

Since the number of pairs of points in the directed landscape is uncountable, there also exist exceptional pairs of points with non-unique geodesics between them. This leads to a very rich fractal structure of non-unique geodesics in the directed landscape. Using the integrable structure, an almost complete description of the network of finite and infinite geodesics has been obtained in this case, see e.g.\ \cite{Bha23b,Dau23} and the references therein. 

{Analogous results also exist for exactly solvable positive temperature polymer models, e.g.\ the \emph{log-Gamma polymer}, see \cite{JRA20,BFS23}.

\medskip
\noindent
{\textbf{Models with additional continuous symmetries.}
As already mentioned, a key obstacle in understanding FPP on $\Z^d$ is the lack of control on the limit shape. To circumvent this, a number of models on (Euclidean) $\R^d$ have been considered with additional continuous rotational symmetry. This ensures that the limit shape is a Euclidean ball. One such model, based on a Poisson point process on $\R^d$, often referred to as the Howard-Newman model was introduced and studied in \cite{HN97,HN01}. Using the rotational invariance of this model, the authors established an analog of Theorem \ref{thm-ftree-intro} and showed that for the forward tree exceptional directions exist and are dense in $S^{d-1}$ \cite[Theorem 1.10]{HN01}. The question of multiplicity of exceptional directions was not considered in \cite{HN01}.}

{Another related result (but for a quite different model) in a non-lattice setting, and the only prior result we are aware of in a hyperbolic background geometry is by Coupier, Flammant and  Tran \cite{CFT23}. There, the authors consider a random tree constructed out of a Poisson process in the hyperbolic plane $\Hyp^2$. This model is called the Hyperbolic Radial Spanning Tree. For this tree, they show that exceptional directions exist and are dense. They also solve the N3G problem in this context.  We should point out, however, that this tree is not a geodesic tree constructed out of a first passage percolation model.}  
}

%
%
%

\subsection{Novelty and structure of the paper}\label{sec-org} 

{As pointed out above, existing results in the literature, mostly in the Euclidean background geometry, focus primarily on either exactly solvable models, or non-lattice models with additional continuous rotational symmetry. Such additional symmetries are typically not present in a Cayley graph. Our results appear to be the first  in the lattice setting for a model which is not exactly solvable, as also the first in hyperbolic FPP. They apply to a  large class of edge length distributions as well as in a higher dimensional setting. Existing literature on non-lattice models only deals with the case where the set of directions is a sphere $S^{d-1}$, whereas we are able to deal with more general boundaries $\pG$ by appealing to deep facts about hyperbolic groups and the technology of Cannon-Thurston maps. Theorem \ref{thm-omni-intro} (3)  is indeed a close
analog of the results described above in the planar exactly solvable set-up, except that we are unable to completely solve the N3G problem. The existing solutions of the (planar Euclidean) N3G problem rely crucially on the integrability of the corresponding models (either via a  connection to TASEP as in \cite{coupier11} or an exact computation of probability of three disjoint geodesics across a trapezium as in \cite{busani2024nonexistencenoncoalescinginfinitegeodesics}). Instead we obtain a weaker but more general bound in  Theorem \ref{thm-omni-intro} (6). We also establish features that appear to be qualitatively new: existence of exceptional directions with arbitrarily large multiplicity (Theorem \ref{thm-omni-intro},(4)) and an upper bound on the Hausdorff dimension of the set of exceptional directions (Theorem \ref{thm-omni-intro},(5)).} 

\medskip
\noindent {\bf Section 2:} Our arguments depend crucially on the interplay between probability and the underlying hyperbolic geometry of the ambient space, often with probabilistic estimates providing key inputs to geometric arguments and vice versa. In Section~\ref{sec-btree}, we  start off by proving Theorem \ref{thm-ftree-intro}, Theorem \ref{thm-btree-intro} and Theorem \ref{thm-bi-intro}: these theorems  show that almost surely the random geodesic trees have {complete} ends. The key technical ingredient here is Lemma \ref{lem-effbt}. This Lemma gives a quantitative estimate on how unlikely  it is for a certain geometric event to occur. This event is given by the following:\\
 There  exists an $\omega$-geodesic between a pair points that is far away from a given point $o$, and the word geodesic between them  passes within a bounded distance of $o$.\\
  The geometric input for Lemma \ref{lem-effbt} is  exponential inefficiency of paths far away from the word geodesic in hyperbolic spaces. 

\medskip
Our results on exceptional directions can be broadly divided into two categories: \\
(i) lower bounds on cardinality and multiplicity of exceptional directions (Theorem \ref{thm-omni-intro} (1), (2), (4))). This is dealt with in Section~\ref{sec-direx}.\\
(ii) upper bounds on the Hausdorff dimension of the set 
of exceptional directions and their multiplicities  (Theorem \ref{thm-omni-intro}  (5), (6)).  This is dealt with in Section~\ref{sec-exceptional-coalesce}. \\
These two types of results rely on different types of arguments. 

\medskip
\noindent {\bf Section 3:}
The lower bounds rely on several deep  facts about hyperbolic groups and their boundaries (this part uses substantially different techniques from previous works \cite{benjamini-tessera, BM19} on hyperbolic FPP which, for the most part, used more elementary properties of hyperbolic spaces). The principal workhorse behind the lower bound arguments is the existence of a random Cannon-Thurston maps (Section~\ref{sec-ct}), a continuous map from the boundary of the random trees to $\pG$ that continuously extends the natural embeddings of the trees into $\Gamma$. Probabilistic estimates (e.g. Lemma \ref{lem-effbt}) go into verifying a necessary and sufficient geometric condition for the existence of Cannon-Thurston maps. The result about complete ends of the geodesic trees (Theorems \ref{thm-ftree-intro} and \ref{thm-btree-intro}) is used to show that the Cannon-Thurston maps are surjective. Once the almost sure existence of such Cannon-Thurston maps is established, the proofs of the lower bounds follow from deep topological facts about the boundaries of hyperbolic groups that are not virtually free and fundamental facts about topological dimensions of compact metrizable spaces. \\
 a) In Section~\ref{sec-direxdense}, we establish density of exceptional directions in $\pG$ (see Theorem~\ref{thm-nonfreedense}). We use Stallings' theorem on ends \cite{stallings-ends},  Dunwoody's accessibility result \cite{dunwoody}, and local connectivity of boundaries of one-ended hyperbolic groups \cite{bes-mess,bowditch-cutpts,swarup-era}.\\
 b) In Section~\ref{sec-direxcard}, we determine the cardinality of exceptional directions (see Theorems~\ref{thm-uncountable} and \ref{thm-countable}). We use Bestvina-Mess's work on boundaries \cite{bes-mess} of hyperbolic groups.\\
 c) In Section~\ref{sec-mult}, we use classical facts from topological dimension theory \cite{engelking}, in particular a theorem of Hurewicz, to establish lower bounds on multiplicity (see Theorem~\ref{thm-mult}).

\medskip
\noindent {\bf Section 4:} In Section~\ref{sec-exceptional-coalesce}, we 
establish
upper bounds.  The proofs  rely on further developing the techniques introduced in \cite{BM19} to study coalescence of geodesics. These results primarily use the technique of dividing the space into half spaces by hyperplanes.   Geometric properties such as exponential divergence of hyperplanes are combined with several probabilistic estimates. Theorem~\ref{t:exceptional} establishes that exceptional directions are of measure zero. In fact, we show that their Hausdorff dimension in $\pG$ is strictly less than that of $\pG$. In Theorem~\ref{t:finite} we establish an upper bound on multiplicity.

\subsection*{Acknowledgments} The authors thank Ofer Busani for useful comments on a previous draft of this article. The authors thank the referee for a number of helpful comments. RB is partially supported by a SERB MATRICS grant  from DST Govt.\ of India (MTR/2021/000093), and DAE project no.\ RTI4001 via ICTS.  MM is partly supported by  a DST JC Bose Fellowship,  and by  the Department of Atomic Energy, Government of India, under project no.12-R\&D-TFR-5.01-0500.  Both authors were supported in part by an endowment of the Infosys Foundation. This project began at the International Centre for Theoretical Sciences (ICTS) during  the program - Probabilistic Methods in Negative Curvature (Code: ICTS/PMNC2019/03). We  gratefully acknowledge the support of ICTS.

\section{Ends of Geodesic Trees and Bigeodesics}
\label{sec-btree}

Our first result is a strengthening of \cite[Proposition 6.3]{BM19} which said that under the  sub-exponential assumption~\ref{assume-subexp}, given $o \in \Gamma$, $C>0$,
 and for a.e.\ $\om \in \omegap$, there exists
$R(\om)>0$  such that the following holds:\\
For every sequence $x_n, y_n \to \infty$ such that $[x_n, y_n]$ passes through 
$N_C (o)$, the $\om-$geodesic $[x_n, y_n]_\omega$ passes
through the $R_\omega-$neighborhood of $o$.We make this statement quantitative by showing minimal such (random) number $R_{\omega}$ will have exponentially decaying tails. 

\begin{prop}\label{prop-effbt} Fix $o \in \Gamma$, and $C > 0$ as above.
Let	$R(\omega)$ be the smallest integer such that for all $x,y$ with $[x,y]\cap N_C(o)\neq \emptyset$, $[x,y]_{\omega}\cap N_{R}(o)\neq \emptyset$. 
 Then there exists $c>0, N_0$ such that $$\P[R(\om)\geq N] \leq \exp(-cN)$$ for $N \geq N_0$.
\end{prop}

\begin{rmk}
 We 
shall use Proposition \ref{prop-effbt} (and the geometric Lemma \ref{lem-effbt} that goes into the proof) multiple times in the paper. In this section, we use it to prove Theorems \ref{thm-ftree-intro}, \ref{thm-btree-intro}, \ref{thm-bi-intro}. We also use it in Section~\ref{sec-exceptional-coalesce}. It is possible that some of its applications in the sequel might be achieved by  refining the arguments in \cite{BM19} (particularly the proofs of   \cite[Theorems 6.6, 6.7]{BM19}).  However,  Proposition \ref{prop-effbt} and Lemma \ref{lem-effbt} lead to a cleaner and more explicit argument. We believe that this result will be useful in future applications as well.
\end{rmk}

Proposition~\ref{prop-effbt} is a consequence of the following more explicit fact.

\begin{lemma}\label{lem-effbt} Let $C>0, o \in \Gamma$ be as in Proposition~\ref{prop-effbt}.
	
	Suppose Assumption~\ref{assume-subexp} holds. Then, given $c>0$, there exists $R_0=R_0(c,C) >0$ such that for $R \geq R_0$ the following holds.\\
Let $A=A(o,C)$ denote the event that there exist $u, v, w \in \Gamma$ such that
	\begin{enumerate}
	\item $w \in [u,v]_\omega$.
	\item $\Pi_{uv} (w) \in N_{C}(o)$, where $\Pi_{uv}$ denotes nearest-point-projection onto $[u,v]$.
	\item $d(w,o) \geq R$.
	\end{enumerate}
	Then $\P(A) \leq e^{-cR}$. 
\end{lemma}

\begin{proof}

\begin{figure}[htbp!]
\includegraphics[width=10cm]{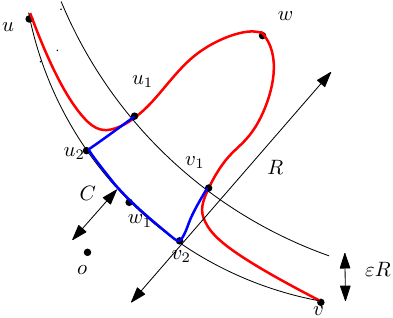}
\caption{Proof of Lemma \ref{lem-effbt}: if the event $A_R$ described in the proof of the lemma is to hold, the FPP length of the red path restricted between $u_1$ and $v_1$ must be smaller than the blue path $[u_1,u_2]\cup [u_2,v_2]\cup [v_2,v_1]$. Since the word length of the red path is much larger than that of the blue path, this is sufficiently unlikely so that we can take a union bound over all possible locations of $u_1,v_1$ and $w$.}
\label{f:effbt}
\end{figure}
For $r\ge R$, let $A_{r}$ denote the event described in the statement of the lemma with the further restriction $d(w,0)=r$. Clearly it suffices to prove $\P(A_{r})\le e^{-cr}$. We shall show that given $c>0$, $\P(A_{R})\le \exp(-cR)$ for all $R$ sufficiently large depending on $c,C$.

For $\ep \in (0,1)$ (to be chosen later), let $N_{\ep R} ([u,v])$ be the $\ep R$ neighborhood of $[u,v]$. Let $u_1$ (resp.\ $v_1$) denote the last point before (resp.\ first point after) $w$ in $[u,v]_\omega\cap N_{\ep R} ([u,v])$. Let $u_2$ (resp.\ $v_2$)
denote  $\Pi_{uv}(u_1)$ (resp.  $\Pi_{uv}(v_1)$). Let $w_1=\Pi_{uv}(w)$. Let $A(l_1,l_2) \subset A_{R}$ denote the subset of $A_{R}$ with $d(w_1,u_2)=l_1, d(w_1,v_2)=l_2$, so that 
$\P(A_{R}) \leq \Sigma_{l_1,l_2\ge 0} \P (A(l_1,l_2))$. We proceed to bound $\P (A(l_1,l_2))$.\\

\noindent {\bf Case 1: $l_1 +l_2 \leq \ep R$.}\\ The 
total set of  possible pairs $u_1, v_1$ such that $u_2, v_2$ satisfies this condition is at most $D^{C+4\ep R}$ ($D^C$ possible choices for $w_1$, $D^{l_1+l_2}$ choices for $(u_2,v_2)$ given $w_1$ and $D^{2\ep R}$ choices for $(u_1,v_1)$ given $(u_2,v_2)$). Further, for any such choice of  $u_1, v_1$,
\begin{enumerate}
\item $d(u_1,v_1) \leq 3 \ep R$,
\item $d(u_1,w) \geq (1-2\ep)R$,
\item $d(v_1,w) \geq (1-2\ep)R$,
\item $\T(u_1,w) + \T (v_1,w) \le \T([u_1,v_1])$
\end{enumerate}
Recall here that $\T$ is the \emph{random} passage time according to $\om \in A$.
In particular, $\T([u_1,v_1])$ denotes the random passage time along a \emph{deterministic} geodesic $[u_1,v_1]$. (Note that $\T([u_1,v_1])$ depends on the geodesic $[u_1,v_1]$ and not just on $u_1, v_1$.)

By \cite[Lemma 5.14]{BM19}, given $c>0$
 we can choose $\kappa>0$ sufficiently small such that (under Assumption~\ref{assume-subexp}), $$\P (\T(u_1,w) + \T (v_1,w) \leq \kappa R) \leq 
 e^{-cR}.$$ 
(We should point out that the statement of \cite[Lemma 5.14]{BM19} is not  phrased exactly this way; however the proof does give this conclusion.) 
 We can thus arrange our choice of $\kappa>0$ so that $c \gg \log(D)$.

 Next,  choosing $\ep > 0$ so that $\kappa \gg \ep$, we have (under Assumption~\ref{assume-subexp}, using the standard concentration inequality \cite[Theorem 2.14]{BM19}) $$\P \big[\T([u_1,v_1]) > \kappa R) \big]\leq 
 e^{-cR}.$$
 
 Next, for $d(w,o)= R$,  summing over all choices  of $u_1, v_1$ (at most  $D^{C+4\ep R}$) and  $w$ (at most $D^{ R}$), we have, 
 $$\sum_{l_1+l_2\le \ep R} \P (A(l_1,l_2)) \leq D^R D^{C+4\ep R}  e^{-cR} \leq  e^{-c'R},$$
 where $c'$ can also be made arbitrarily large (since  $c \gg \log(D)$).\\
 
 \noindent {\bf Case 2: $l_1 +l_2 \geq \ep R$.}\\ Here, $d(u_1, v_1) \leq 2\ep R + l_1 +l_2\le 3(l_1+l_2)$, whereas $[u_1,v_1]_\omega$ has at least 
 $(l_1 +l_2)e^{\alpha  \ep R}$ edges by exponential inefficiency of paths traveling outside $N_{\ep R} ([u,v])$ (see \cite[Chapter III.H.1]{bh-book}
 or \cite[Lemma 7.4]{BM19}). Thus, for $A_{l_1,l_2}$ to occur in this case, the number of edges in $[u_1,v_1]_{\omega}$ is at least $e^{\alpha\epsilon R}/3$ times $d(u_1,v_1)$. The proof of \cite[Lemma 4.3]{BM19}  gives that there exists $R_0>0$ such that for all $u,v\in \Gamma$ and for $R>R_0$, the probability that the number of edges in $[u,v]_{\omega}$ exceeds $Rd(u,v)$ is upper bounded by $\exp(-c\sqrt{R}d(u,v))$. It follows that for fixed $u_1,v_{1}$ this probability is upper bounded by  $\exp(-c\sqrt{3^{-1}e^{\alpha\epsilon R}}d(u_1,v_1))$ for some $c>0$.
 
 {Observe now that in this case if $R$ is sufficiently large  we also have $d(u_1,v_1)\ge c_1(l_1+l_2)$ for some some $c_1>0$. Indeed, $$d(u_1, v_1) \geq d(u_1,u_2) + d(u_2,v_2) + d(v_2,v_1)-8 \delta$$ whenever 
  	$ d(u_2,v_2) = l_1+l_2$ is large enough.} Summing over all possible choices of $u_1$ and $v_1$ (at most $D^{C+l_1+l_2+2\epsilon R}$ many choices) we get  
 $$\P(A(l_1,l_2))\le  D^{C+l_1+l_2+\epsilon R}\exp(-c'\sqrt{\frac{e^{\alpha\epsilon R}}{3}}(l_1+l_2))$$
 for some $c'>0$.Summing over $(l_1,l_2)$ with $l_1+l_2\ge \epsilon R$ get
 $$\sum_{l_1+l_2\ge \ep R} \P (A(l_1,l_2)) \leq   e^{-cR}$$
 where $c$ can be made arbitrarily large by choosing $R$ large enough. Combining the two cases, we get the lemma. 
\end{proof}

We shall now show how Proposition \ref{prop-effbt} follows from Lemma \ref{lem-effbt}. We first need the following observation. 

\begin{rmk}\label{rmk-cs} {There exists (universal) $k_0$ such that for any $\delta-$hyperbolic $\Gamma$, 
 any geodesic $[x,y]\subset \Gamma$, and any (necessarily continuous) edge-path $\sigma_{xy}$ in  $\Gamma$ joining $x,y$, a nearest-point projection $\pi_{xy}$ of $\sigma_{xy}$ to $[x,y]$ is \emph{coarsely surjective} in the following sense:\\ for all $z \in  [x,y]$ there exists $w\in \sigma_{xy}$ and
 $z_0 \in  [x,y]$ such that
 \begin{enumerate}
 \item $\pi_{xy}(w) = z_0$,
 \item  $d(z,z_0) < k_0\delta$.
 \end{enumerate} }
\end{rmk}

\begin{proof}[Proof of Proposition \ref{prop-effbt}]
{For $k_0$ as in Remark \ref{rmk-cs} observe that  for any $x,y$ with $[x,y]\cap N_C(0)\ne \emptyset$, there must exists a point $w\in [x,y]_{\omega}$ such that $\pi_{xy}(w)\in N_{C+k_0\delta}(o)$. Clearly, on the complement of the event $A(o,C+k\delta)$ in Lemma \ref{lem-effbt}, $[x,y]_{\omega}$ intersects $N_R(0)$ for all such $x,y$ and therefore we have $R(\omega)\le R$. Proposition \ref{prop-effbt} follows immediately. This, in particular also shows that $R(\omega)$ is finite almost surely.} 
\end{proof}

Since $G$ acts transitively on $\Gamma$, Proposition~\ref{prop-effbt} and Lemma~\ref{lem-effbt} hold with  \emph{uniform} $N_0, R_0$ respectively
\emph{independent of $o$}. In Proposition \ref{prop-effbt}, the constant $c$ is also uniform. Here, as usual, we identify $G$ with the vertex set of $\Gamma$.

Let $1 \in G$ denote the identity element. Fix $C>0$ as in Lemma~\ref{lem-effbt}.
Further, fix $\ep \in (0,1)$.
For $v \in G$, let $\AAA_v$ denote the event that for all $u \in G$ satisfying
$ N_C(v) \cap [u,1] \neq \emptyset$, and any $w \in [u,v]_\omega$ satisfying 
$\Pi_{u1} (w) \in N_C(v)$, we have $d(w,v) \leq \ep d(1,v) $. 
Further, let $\AAA_n = \cap_{v \in \partial B_n(1)}\AAA_v$.

\begin{lemma}\label{lem-intnAn}
$\P [\liminf_{n \to \infty}\AAA_n ] = 1$.
\end{lemma}

\begin{proof}
For $c$ such that $c\gg \frac{1}{\ep}$, and $ C \geq 0$, 	let $R_0$ be as in Lemma~\ref{lem-effbt}.
For $d(v,1) > \frac{R_0}{\ep}$, $\P(\AAA_v^c) \leq e^{-c\ep d(v,1)}$ by Lemma~\ref{lem-effbt}.
Taking a union bound over $v \in \partial B_n(1)$ ($n>R_0$),
$\P[\bigcup_{v \partial B_n(1)} \AAA_v^c] \lesim e^{-c' n}$ for some $c'>0$, since $c\ep \gg 1$, and the ball of radius $n$ has at most $D^n$ vertices. The
Lemma now follows from the Borel-Cantelli Lemma.
\end{proof}

\noindent {\bf A geometric generalization:} In the proof of Lemma~\ref{l:finite}, we shall need a geometric generalization of Proposition~\ref{prop-effbt} and Lemma~\ref{lem-effbt} without using the group-invariant structure of the ambient Cayley graph $\Gamma$. The particular context we shall need the analogs in is that of quasiconvex subsets of $\Gamma$. However, we prefer to isolate the exact geometric features necessary for the argument in Lemma~\ref{lem-effbt} (and hence 
Proposition~\ref{prop-effbt}) to go through. We thus start with a $\delta-$hyperbolic graph $X$.

Notice that for Case 1 of the proof of Lemma~\ref{lem-effbt}, we used the following:
\begin{enumerate}
\item An upper bound $D$ on the valence of every vertex of $X$.
\item \cite[Lemma 5.14]{BM19}:  the proof uses the upper bound $D$ and 
\cite[Lemma 4.3]{BM19}.
\item The general concentration inequality \cite[Theorem 2.14]{BM19}.
\end{enumerate}

For Case 2 of the proof of Lemma~\ref{lem-effbt}, we used the following:
\begin{enumerate}
	\item An upper bound $D$ on the valence of every vertex of $X$.
	\item Exponential inefficiency of paths traveling far away from geodesics  \cite[Chapter III.H.1]{bh-book}.
	\item \cite[Lemma 4.3]{BM19}: as pointed out in the paragraph preceding the statement of \cite[Lemma 4.3]{BM19} in that paper, this is a very general statement and does not need the hypothesis that $X$ is a Cayley graph. (In fact, neither 4.3 nor 5.14 really use hyperbolicity, but we shall not need this here.)
\end{enumerate}
Finally, note that the proof of Lemma~\ref{lem-effbt} is independent of the origin $o$, in the sense that the constant $R_0$ there depends only on $\delta, D$ and the constant $C>0$ in its statement.
We thus conclude:

\begin{cor}\label{cor-effbt}
Let $X$ be a $\delta-$hyperbolic graph with a uniform upper bound $D$ on the degree of every vertex. Then for $o\in X$, Lemma \ref{lem-effbt} holds uniformly in $o$, i.e., for $C,c>0$ there exists $R_0>0$ such that for all $o\in X$, the event $A=A(o,c,R)$ satisfies $\P(A)\le \exp(-cR)$.   
\end{cor}

\subsection{Forward and backward trees with complete ends}\label{sec-btstr}
Let $\AAA_n$ be as in the discussion preceding Lemma~\ref{lem-intnAn}.
\begin{defn}\label{def-intnAn}
Define $\AAA_\infty:=\liminf_{n \to \infty}\AAA_n$.
\end{defn}

\begin{lemma}\label{lem-backgeo}
Let $\om \in \AAA_\infty$. Then, for any $\xi \in \partial G$ and any word geodesic
$[1, \xi)$ defined by the sequence of vertices $1, v_1,\cdots, v_n, \ldots$, the only accumulation point of $[1,v_n]_\omega$  on $\partial G$ is $\xi$.
\end{lemma}

\begin{proof} Given $\omega$,  let $N (=N(\omega)) \in \natls$ be such that $\om \in \bigcap_{n \geq N} \AAA_n$. 
Then the edges in $\cup_n [1,v_n]_\omega$ lying outside $B_N(1)$ are contained in the cone $\CC(\ep, \xi)$ given by $\CC(\ep, \xi)=\bigcup_{i \geq N} B_{\ep i} (v_i) $. Since the only accumulation point of $\CC(\ep, \xi)$ is $\xi$ by Claim~\ref{claim-cone} below, the Lemma follows.
\end{proof}

\begin{claim}\label{claim-cone}
The only accumulation point of $\CC(\ep, \xi)$ in $\partial G$ is $\xi$.
\end{claim}

\begin{proof}
Suppose $\xi' \in \partial G$ is also accumulated upon by $\CC(\ep, \xi)$. Then
there exists an open neighborhood $\BB(\xi') \subset \Gamma \cup \partial G$ such that
\begin{enumerate}
\item $\BB(\xi')$ is disjoint from an open neighborhood of $\xi$ in
$\Gamma \cup \partial G$,
\item $\BB(\xi')$ contains infinitely many points of $\CC(\ep, \xi)$.
\end{enumerate}
 The nearest-point-projection
of $  \BB(\xi') \setminus \partial G$ onto $[1,\xi)$ is a bounded geodesic segment
$[a,b]$. Hence there exist infinitely many points $w_n$ in $(\BB(\xi') \setminus \partial G) \cap \CC(\ep, \xi)$ converging to $\xi'$ whose nearest-point-projection lies in $[1,b]$. Since  $[1,b]$ has finitely many vertices, we can find $b_0 \in [1,b]$ and pass to a subsequence if necessary to assume that the  nearest-point-projections of all the $w_n$'s is $b_0$.
But this forces the sequence $w_n$ to have $d(1,w_n) \to \infty$; in particular,
$w_n$'s must eventually leave the ball  of radius $\ep d(b_0,1)$ about $b_0$. This contradicts the assumption that $w_n$'s lie in
$\CC(\ep, \xi)$ and completes the proof. 
\end{proof}

{ 

The following definition demands that the forward tree $F(1,\omega)$ has no `effective' dead-ends, and moreover extends forwards in all possible directions.

\begin{defn}\label{def-ftreefullend}
	We shall say that the forward tree $F(1,\omega)$  has 
	\emph{complete ends} if for all $\xi\in \pG$, there exists some $\omega-$geodesic  $[1, \xi)_\om$ such that $[1, \xi)_\om \subset F(1,\omega))$.
\end{defn}


The following Theorem simultaneously refines and strengthens \cite[Theorem 6.6]{BM19}
and \cite[Theorem 6.7]{BM19}.
\begin{theorem}\label{prop-ftree} 
	There exists a full measure subset $\Omega_1\subset \Omega$ such that  for all $\om \in \Omega_1$,  the forward tree
	$F(1,\omega)$ has complete ends.
\end{theorem}

\begin{proof}
Note that we fix  $1 \in \Gamma$ as a root for $F(1,\omega)$. 
 The first part of Theorem~\ref{bmthmcoalesce} along with the fact that existence of geodesic bigons is a zero probability event shows that
	$$F(1,\omega) = \cup_{g \in G} [1,g]_\omega$$ is a tree for a full measure subset $\Omega_1 \subset \Omega$.  This is the analog of Corollary~\ref{cor-btree} and is already incorporated into Definition~\ref{def-fwdtree}.

The proof of Lemma~\ref{lem-backgeo} now shows that each infinite geodesic ray in
	$F(1,\omega)$ starting at 1 has a unique point of accumulation on $\pG$.
Complete ends (in the sense of Definition~\ref{def-ftreefullend}) now follows.
\end{proof}
}
We proceed with the analogous statements for backward trees. This is somewhat more involved.
The following definition demands that the backward tree $T(\xi,\omega)$ (Definition~\ref{def-btree}) has no `effective' dead-ends, and moreover extends backwards in all possible directions.
\begin{defn}\label{def-btreefullend}
We shall say that a backward tree $T(\xi,\omega)$  has 
\emph{complete ends} if for all $\xi'\neq \xi$, there exists 
a bi-infinite geodesic $(\xi',\xi)_\om \subset T(\xi,\omega) $.
\end{defn}

{

There are a couple of equivalent ways of saying when a backward tree $T(\xi,\omega)$  has 
{complete ends}. Let us use 
 the notation  $(\xi',p]_\om$ 
to  suggest that the semi-infinite geodesic  $(\xi',p]_\om$ is parametrized by $(-\infty, 0]$  (note that 
$(\xi',p]_\om$ equals $[p, \xi')_\om$ as unparametrized geodesics). 
Next, fix a root vertex $o \in \Gamma$. 
Then 
we observe that $T(\xi,\omega)$  has 
{complete ends} if and only if for all $\xi'\neq \xi$, there exists  
$p \ in [o, \xi)_\omega$ and some $\omega-$geodesic $(\xi',p]_\om \subset
T(\xi,\omega)$. Further, this is equivalent to the existence of  a sequence $x_{n}\to \xi'$ such that $[x_n,\xi)_{\omega} \subset [x_m,\xi)_{\omega}$ for $n<m$. This forces the increasing limit 
$\cup_n [x_n,\xi)_{\omega}$ to be a bi-infinite geodesic $(\xi',\xi)_\om$.
Further, by definition of $T(\xi,\omega)$, $(\xi',\xi)_\om \subset T(\xi,\omega) $.

\begin{theorem}\label{thm-bi} There exists a full measure subset $\Omega_2 \subset \Omega$ such that  for all $\om \in \Omega_2$,
and all $\xi_1 \neq \xi_2 \in \partial G$,
 there exists a bi-infinite $\omega-$geodesic $(\xi_1$,$\xi_2)_\omega$.
\end{theorem}

\begin{proof}
Fix $C \in \natls$. Let $\xi_1, \xi_2 \in \pG$ be any pair of points satisfying
$\langle \xi_1, \xi_2 \rangle_1 \leq C$. Also, let $x_{1n} \to \xi_1$ and 
 $x_{2n} \to \xi_2$ be any two sequences in $\Gamma$ such that 
 $\langle x_{1n}, x_{2m} \rangle_1 \leq 2C$ for all $m, n$. This can always be arranged by dropping the first few terms in $\{x_{1n}\}$ and 
 $\{x_{2m}\}$ if necessary.

  Let $A(C, R)$ denote the event that for some  choice of $\xi_1, \xi_2 \in \pG$ with $\langle \xi_1, \xi_2 \rangle_1 \leq C$ and any  two sequences $\{x_{1n}\}$ and 
  $\{x_{2m}\}$ as above, the following holds:   there exists $w \in [x_{1n}, x_{2m}]_\omega$ such that
  \begin{enumerate}
  \item the nearest point projection of $w$ onto (a deterministic geodesic) $[x_{1n}, x_{2m}]$ lies in $B_{2C} (1)$,
  \item $d(w,1) \geq 2C + R$.
  \end{enumerate}
  
  Then Lemma~\ref{lem-effbt} shows that  given $c>0$, there exists $R_0 >0$ such that for $R \geq R_0$, 
 $\, \P\big(A(C, R)\big) \leq e^{-cR}$. Hence, there exists a full measure set $\Omega (C) \subset \Omega$ such that for all $\omega \in \Omega (C) $, there exists $R(\omega) \in (0,\infty)$ (see Proposition~\ref{prop-effbt}) such that for all $\xi_1, \xi_2, x_{1n}, x_{2m}$ as above, $[x_{1n}, x_{2m}]_\omega$ passes through the 
 $R(\omega)-$ball about $1 \in \Gamma$. Since $R(\omega)$ is finite, 
 there exists a point $w(\xi_1,\xi_2)$ contained in the $R(\omega)-$ball,
 and  a subsequence  $[x_{1n}', x_{2n}']_\omega$ such that 
 $w(\xi_1,\xi_2) \in [x_{1n}', x_{2n}']_\omega$. 
 
 Now, by Lemma~\ref{lem-backgeo}, 
 \begin{enumerate}
 \item $  [w(\xi_1,\xi_2), x_{2n}']_\omega$
 converges to a semi-infinite geodesic $  [w(\xi_1,\xi_2), \xi_2)_\omega$, and 
 \item $  [x_{1n}', w(\xi_1,\xi_2)]_\omega$ converges to a semi-infinite geodesic $  (\xi_1, w(\xi_1,\xi_2)]_\omega$,
 \end{enumerate} 
 so that $[x_{1n}', x_{2n}']_\omega$  converges to a bi-infinite geodesic
  $  (\xi_1, ,\xi_2)]_\omega$. 
  
  Finally, let $C$ range over the natural numbers $\natls$, and set
  $\Omega_2 = \cap_C \Omega (C) $ to obtain the required  full measure subset.
  \end{proof}

As a consequence of Theorem~\ref{thm-bi}, we have the following:
\begin{theorem}\label{thm-btree}
	For all $\xi \in \partial G$,
there exists a full measure subset $\Omega_\xi\subset\Omega$ such that  for all $\om \in \Omega_\xi$,  the backward tree
$T(\xi,\omega)$ has complete ends.
\end{theorem}

\begin{proof} Let $\Omega (\xi)$ denote {the}
  full measure subset of $\Omega$ furnished by Corollary~\ref{cor-btree} so that for all $\omega \in \Omega (\xi)$, the backward tree
 $T(\xi,\omega)$ is well-defined (see  Definition~\ref{def-btree}).

Now, set $\Omega_\xi = \Omega (\xi) \cap \Omega_2$, where $\Omega_2$ is as in 
Theorem~\ref{thm-bi}. Observing that $(\xi_1,\xi)_\omega \subset T(\xi,\omega)$
for all $\om \in \Omega_\xi$, it follows that $T(\xi,\omega)$ has complete ends.
\end{proof}
}

\section{Exceptional directions: geometric and topological estimates}\label{sec-direx}
In this section and the next, we shall deal with exceptional directions for random forward and backward trees. The proofs are almost identical and need only minor terminological change. We refer the reader to Section~\ref{sec-treeqns} for a conceptual framework using KAN decompositions of semi-simple Lie groups that explains 
why the proofs  for random forward and backward trees mostly agree.
We continue with Assumption~\ref{assume-subexp} in this section. Given the existence of forward and backward random trees with complete ends from Theorems~\ref{prop-ftree} and ~\ref{thm-btree}, the purpose of this section is to investigate if \emph{exceptional directions} exist in $F(1,\omega)$ and $T(\xi,\omega)$.

\subsection{Random Cannon-Thurston Maps}\label{sec-ct}
The notion of a Cannon-Thurston map originated in work of Cannon and Thurston in \cite{CT,CTpub}, and was extended to the  context of
(Gromov) hyperbolic metric spaces in \cite{mitra-trees} (see \cite{mahan-icm} for a survey of developments). {Recall that a function $M: \natls \to \natls$ is said to be \emph{proper} if $n \to \infty$ implies
$M(n) \to \infty$.} We start with the following.

\begin{defn}\label{def-proper}
	We say that an embedding $i:(Y,y_0,d_Y) \to (X,x_0, d_X)$ of metric spaces with  base-points $y_0, x_0$ is proper if  there exists a proper function $M: \natls \to \natls$ such that  if $d_Y(y_0,y) \geq n$, then $d_X(x_0,i(y)) \geq M(n)$.
\end{defn}

\begin{defn}\label{def-ct}
Let $(X, d_X)$ and  $(Y, d_Y)$  be  hyperbolic metric spaces and let $i: (Y,y_0) \to (X,x_0)$ denote a proper embedding. Let $\partial X, \, \partial Y$ be their (Gromov) boundaries.
Also, let $\hhat X$ and $\hhat Y$ denote their Gromov compactifications.
We say that the triple $(X,Y,i)$ admits a Cannon-Thurston map, if
$i$ extends continuously to 
$\hhat i: \hhat Y \to \hhat X$. The continuous extension, if it exists, is referred to as a \emph{Cannon-Thurston map} for 
$i: Y \to X$.
\end{defn}

 In the Lemma below, $[a,b]_Y$ will denote a geodesic in $Y$ joining $a, b \in Y$ and $[i(a),i(b)]_X$ will denote a geodesic in $X$ 
 joining $i(a),i(b) \in X$.
 \begin{lemma}\cite[Lemma 3.1]{mahan-icm} \label{lemma-ctsuff} Let $i : (Y,y) \to (X,x)$ be a proper map between (based) Gromov-hyperbolic spaces.
 	A  Cannon-Thurston map
 	$\hat{i}: \hhat{Y} \to \hhat{X}$
 	exists if  and only if the following holds: \\
 	There exists a non-negative proper function  $M: \natls \to \natls$, such that 
 	if $[a,b]_Y$  lies outside an $N$-ball
 	around $y$ in $Y$, then $[i(a),i(b)]_X$
 	lies outside the $M(N)$-ball around 
 	$x=i(y)$ in $X$.
 \end{lemma}

Note that
in Lemma~\ref{lemma-ctsuff}, the preliminary assumption is that $i$ is   a proper map of based spaces; thus we demand  that $d_Y(y_1,y_0) \to \infty$ implies
$d_X(i(y_1),x_0) \to \infty$.

\begin{defn}\label{def-2proper}
	We say that an embedding $i: (Y,y_0, d_Y) \to (X,x_0, d_X)$ between hyperbolic spaces with base-points is 2-proper if  there exists a proper function $M': \natls \to \natls$ such that for all $y_1, y_2 \in Y$,
	if $d_Y(y_0, [y_1,y_2]_Y) \geq n$, then $d_X(x_0, [i(y_1),i(y_2)]_X) \geq M'(n)$.
\end{defn}

{\begin{rmk}\label{rmk-2proper}
		Note that Definition~\ref{def-proper} gives only a {proper} function
		$M(n)$ comparing distances between pairs of points  measured with respect to $d_X$ and $d_Y$.
		
		On the other hand, in Definition~\ref{def-2proper}, the image 
		$i([y_1,y_2]_Y)$  of the geodesic 
		$[y_1,y_2]_Y$ under $i$, and the geodesic $[i(y_1),i(y_2)]_X$  have nothing really
		to do with each other. Properness in the sense of Definition~\ref{def-proper} will only control the distance of $i([y_1,y_2]_Y)$
		from $x_0$, but will say nothing about $d_X(x_0,[i(y_1),i(y_2)]_X)$.
		
		Note also that Lemma~\ref{lemma-ctsuff} says equivalently that  an embedding $i: (Y,y_0,d_Y) \to (X,x_0,d_X)$ between hyperbolic spaces admits a Cannon-Thurston map if and only the embedding is 2-proper.
\end{rmk}}

\begin{defn}\label{def-randomtree}
Let $(T(\xi,\omega),d^s_\omega)$ denote the backward random tree with the intrinsic simplicial metric $d^s_\omega$ inherited from $\Gamma$, i.e.\ $d^s_\omega (x,y)$ equals the minimum length of an edge path (with unit weights) 
in $T(\xi,\omega)$  from $x$ to $y$. (To prevent confusion with the random metric $d_\omega$  on $\Gamma$, we use the superscript $s$ to emphasize that we are using the \emph{simplicial metric} on the random 
tree  $T(\xi,\omega)$.)

Similarly, let $(F(1,\omega),d^s_\omega)$ denote the forward random tree with the intrinsic simplicial metric.

For both $(T(\xi,\omega),d^s_\omega)$ and  $(F(1,\omega),d^s_\omega)$, we set $1$ to be a preferred base-point.
\end{defn}

\begin{rmk}
The  choice of base point  in Definition~\ref{def-randomtree} is arbitrary, but it is important that a choice be made so as to apply Lemma~\ref{lemma-ctsuff} in Theorem~\ref{thm-12proper} below.
\end{rmk}

{

Below, $(F(1,\omega),d^s_\omega)$ and $(T(\xi,\omega),d^s_\omega)$ will take the place of $(Y,d_Y)$ in Lemma~\ref{lemma-ctsuff} above
and $\Gamma$ with its word metric will take the place of $X$.
{ As noted above, 1 will be the preferred base-point.}

Let  $\Omega_1 \subset \Omega$ be the full measure subset for which there are no finite $\omega-$geodesic bigons, i.e.\ it is the subset of $\Omega$ where there exist unique $\omega-$geodesics between any two points in $\Gamma$. 

\begin{lemma}\label{lem-proper} Let  $\Omega_1 \subset \Omega$ be as above.\\
	1) For $\omega \in \Omega_1$
the forward random tree $(F(1,\omega), d^s_\omega)$	 is properly embedded in $(\Gamma,d)$ with respect to the base-point 1. \\
	2) Let $\Omega_\xi$ be as in Theorem~\ref{thm-btree}.
Then
for all $\xi \in \pG$ and $\omega$ in  $\Omega_1\cap \Omega_\xi$, {$(T(\xi,\omega),d^s_\omega)$}  is properly embedded in $(\Gamma,d)$ with respect to the base-point 1. 
\end{lemma}

\begin{proof} {In both cases ,we need to show that 
$\max_{w\in B_n(1)} d_{\om}^{s}(1,w)<N(\omega)<\infty$. Since it is only a finite maximum the result follows in the case of $F(1,\omega)$ from the fact that for each  $w\in B_n(1)$, there is a unique (finite length) geodesic from $1$ to $w$. 

In the case of $T(\xi,\omega)$, $[1,\xi)_\om \cup [w, \xi)_\om$ is a subtree
of $T(\xi,\omega)$; so that  there exists $z\in \Gamma$ such that $d_{\om}^s(1,w)=d_{\om}^s(1,z)+d_{\om}^s(z,w)$.
Again,  taking a finite maximum  over $w\in B_n(1)$ we obtain a {proper} function as per Definition~\ref{def-proper}. (We  note in conclusion that the {proper} function depends on $\omega$.) }
\end{proof}

{
\begin{rmk}
Note that the quantification in Lemma~\ref{lem-proper} is such that the full measure subset $\Omega_1$ is independent of $\xi$ (this is irrelevant for forward trees anyway). In particular, we are not using the conclusion of Theorem~\ref{thm-btree} for this part. For $\omega \in \Omega_1$, the union
$\cup_g [g, \xi)_\omega$ does not have finite $\omega-$geodesic bigons. However, this
only ensures that $\cup_g [g, \xi)_\omega$ is a forest; in particular it allows
more than one disjoint $\omega-$geodesic to accumulate at $\xi$. To upgrade 
$\cup_g [g, \xi)_\omega$ to a tree, we need to intersect $\Omega_1$ with $\Omega_{\xi}$.

Note also that in both  cases of Lemma~\ref{lem-proper}, the embedding in question is given by the identity on the vertices of $\Gamma$. 
\end{rmk}
}

\begin{defn}\label{def-AC}
{Recall from Proposition \ref{prop-effbt} that $R(\omega)=R_{1,C}(\omega)$ denotes the smallest integer such that for any $x,y$ with $[x,y]\in N_C(1)$, $[x,y]_{\omega}$ intersects $N_{R_{1,C}(\omega)}(1)$. Let $E_{\infty}(C)$ denote the event that $R_{1,C}(\omega)=\infty$. We note that Proposition \ref{prop-effbt} implies that $E_{\infty}(C)$ is a zero probability event.}
\end{defn}

In Theorem~\ref{thm-12proper} below, we shall use $i_\omega^T$ to denote the inclusion of $T(\xi,\omega)$ into $\Gamma$ for any $\xi$ (though strictly speaking, the inclusion map also depends on $\xi$).

\begin{theorem}\label{thm-12proper} We fix the base-point $1$.
There exists a full measure subset $\Omega_0 \subset \Omega$ such that
for all  $\omega \in \Omega_0$, 
\begin{enumerate}
	\item $i_\omega^F: F(1,\omega) \to \Gamma$ is  proper in the sense of Lemma~\ref{lem-proper}, 
	\item $i_\omega^F: F(1,\omega) \to \Gamma$ is uniformly 2-proper in the sense of Definition~\ref{def-2proper}.
\end{enumerate}
Hence for  all  $\omega \in \Omega_0$, 
$i_\omega^F: F(1,\omega) \to \Gamma$ admits a Cannon-Thurston map. 

Further, given $\xi \in \pG$, and all $\omega \in \Omega_0 \cap \Omega_{\xi}$, 
\begin{enumerate}
\item $i_\omega^T: T(\xi,\omega) \to \Gamma$ is  proper in the sense of Lemma~\ref{lem-proper}, 
\item $i_\omega^T: T(\xi,\omega) \to \Gamma$ is  2-proper in the sense of Definition~\ref{def-2proper}.
\end{enumerate}
Hence for  all $\xi \in \pG$, and $\omega \in \Omega_0 \cap \Omega_{\xi}$, 
$i_\omega^T: T(\xi,\omega) \to \Gamma$ admits a Cannon-Thurston map.
Finally, (after possibly passing to a further full measure subset)
 the (boundary values of the) Cannon-Thurston maps $\partial i_\omega^F: \partial F(1,\omega) \to \pG$ and $\partial i_\omega^T: \partial T(\xi,\omega) \to \pG$ are surjective.
\end{theorem}

\begin{proof}  Properness of $i_\omega^F$ and $i_\omega^T$ follow from Lemma~\ref{lem-proper}. {Thus the preliminary properness hypothesis of Definition~\ref{def-ct} is satisfied for $i_\omega^F$ and $i_\omega^T$.}

	 For $2$-properness, let $C \in \natls$. For each such $C$,  Definition \ref{def-AC} (and Proposition \ref{prop-effbt} referred there) shows that $E_\infty(C)$ has probability zero. Hence, $$E_\infty:=\bigcup_{C \in \natls}E_\infty(C)$$ also  has probability zero.
	 Let $\Omega_1$ be as in Lemma~\ref{lem-proper}, and $\Omega_0 = \Omega_1 \setminus E_\infty$. 

We show below that if $\omega \in  \Omega_0 $, then 
$i_\omega^F$ is 2-proper. Also, given $\xi$, if $\omega \in  \Omega_0\cap \Omega_{\xi} $, then we show that  $i_\omega^T$ is
	2-proper. 
	
For ease of exposition, we  first check for  $i_\omega^F$. For $\omega\in \Omega_0$ define the 2-{proper} function $M(C)=M(C,\omega)=R_{1,C}(\omega)$ as in Definition \ref{def-AC}. 
For any $[a,b]_\omega \subset [a,1]_\om \subset F(1,\omega)$,
	if $[a,b]_\omega$ lies outside the  $M=M(C,\omega)$ ball about $1$ in $\Gamma$, then (any deterministic geodesic) $[a,b]$ lies outside the $C-$ball about $1$ in $\Gamma$.

	Next, observe that a geodesic $\gamma_{ab} $
	in the intrinsic simplicial metric on
	$ F(1, \om)$ joining arbitrary $a, b \in  F(1,\omega)$  satisfies 
	$\gamma_{ab} \subset [a,1]_\om\cup [b,1]_\om$ so that there exists $f \in \gamma_{ab} $ such that the geodesic $\gamma_{af} $ 	in the intrinsic simplicial metric on
	$ F(1, \om)$ joining  $a, f$ satisfies  $\gamma_{af}= [a,f]_\omega \subset [a,1]_\om \subset F(1,\omega)$. Similarly, the geodesic $\gamma_{bf} $ 	in the intrinsic simplicial metric on
	$ F(1, \om)$ joining  $b, f$ satisfies  $\gamma_{bf}= [b,f]_\omega \subset [b,1]_\om \subset F(1,\omega)$. Also, $\gamma_{ab} =\gamma_{af} \cup \gamma_{bf} $.
	If $\gamma_{ab}$ lies outside the  $M=M(C,\omega)$ ball about $1$ in $\Gamma$, then
	so do $[a,f]_\omega $ and $[b,f]_\omega$.
	By the argument in the previous paragraph, (any deterministic geodesics) $[a,f], [b,f]$ lies outside the $C-$ball about $1$ in $\Gamma$.
	Since $[a,b] \subset N_\delta ([a,f]\cup [b,f])$, $[a,b]$ lies outside the $(C-\delta)$ball about $1$ in $\Gamma$. This completes the proof that $i_\omega^F$ is 	2-proper. 
	
	We turn now to the proof that for fixed $\xi\in \pG$, $\omega \in \Omega_0\cap \Omega_{\xi}$,  $i_\omega^T$ is
	2-proper.  For $a, b \in T(\xi,\omega)$, satisfying $[a,b]_\omega \subset [a,\xi)_\om \subset T(\xi,\omega)$, the argument is now the same as in the first case for $F(1,\omega)$ where we assumed that
	$[a,b]_\omega \subset [a,1]_\om \subset F(1,\omega)$.
	In general, a geodesic $\gamma_{ab} $
	in the intrinsic simplicial metric on
	$T(\xi,\omega)$ satisfies  $\gamma_{ab} \subset [a,\xi)_\om \cup [b,\xi)_\om$, and decomposes as a union $[a,f]_\omega \cup [b,f]_\omega$, where 
	$[a,f]_\omega \subset [a,\xi)_\om$ and 	$[b,f]_\omega \subset [b,\xi)_\om$.
	Now, repeating the argument in the second (general) case for $ F(1, \om)$ completes the proof that $i_\omega^T$ is 2-proper.

The  two claims in the proposition for both $i_\omega^F$ and  $i_\omega^T$
 follow. The existence of a  Cannon-Thurston map now follows from Lemma~\ref{lemma-ctsuff} (see the last paragraph of Remark~\ref{rmk-2proper}).
 
 We finally turn to surjectivity. We  pass to a further full measure subset
 if  necessary as dictated by Theorem~\ref{prop-ftree} or Theorem~\ref{thm-btree}.

Surjectivity of the Cannon-Thurston map to $\pG$ for forward trees follows from Theorem~\ref{prop-ftree} that guarantees existence of
forward geodesics to \emph{all $ \xi$}. 

Surjectivity of the Cannon-Thurston map to $\pG$ for backward trees
$T(\xi,\omega)$
 follows from Theorem~\ref{thm-btree}, since there are 
 backward geodesics to \emph{all $\xi'\neq \xi$}. 
\end{proof}

\begin{remark}\label{rmk-caveat}
A caveat is in order. Theorem~\ref{thm-12proper} ensures the existence of a Cannon-Thurston map  for $T(\xi,\omega)$
for a  full measure set $\Omega_\xi$ that depends on $\xi$. Note also that
the  argument for existence only uses that $T(\xi,\omega)$ is a tree, i.e.\ Corollary~\ref{cor-btree} (and hence Theorem~\ref{bmthmcoalesce}) but not
Theorem~\ref{thm-btree}. Theorem~\ref{thm-btree} gets used at the last step to prove surjectivity.
An analogous statement  using Theorem~\ref{prop-ftree} holds for forward trees. Henceforth when we refer to a \emph{full measure subset} of $\Omega$, we shall assume implicitly that such a full measure subset is dependent on $\xi \in \pG$ in the case of the backward tree. Since the forward tree has no $\xi$ in it,
the independence of this full measure subset on $\xi \in \pG$ in the case of the forward tree is automatic.
This is, in particular, relevant in Theorems~\ref{thm-nonfreedense}, \ref{thm-uncountable},
\ref{thm-countable}, \ref{thm-mult} below.\\
\end{remark}

\noindent {\bf Aside: Cannon automata and Cannon-Thurston maps:}\\
The existence of a surjective Cannon-Thurston map from  $\partial F(1,\omega)$
onto $\pG$, and Theorem~\ref{thm-12proper} have a deterministic analog
in the literature. 

The geodesic language in a hyperbolic group is \emph{regular}, i.e.\ there exists a finite state automaton $\GG$ whose associated regular language is the geodesic language in $\Gamma$. This is originally due to Cannon (see \cite{calegari-notes} for an exposition).
Since $\GG$ is a finite, rooted graph, its universal cover $\til \GG$ is a rooted directed tree with root at 1, say.
In 
\cite[Lemma 3.5.1]{calegari-notes}, Calegari notes that the evaluation
$ev: \til \GG \to \Gamma$ extends to a continuous, surjective, bounded-to-one
map $\partial{ev}: \partial \til \GG \to \pG$.

\subsection{Existence and Density of exceptional directions}\label{sec-direxdense}
In this section, we shall use the Cannon-Thurston maps constructed above to show that \emph{exceptional directions} exist and are dense in $\partial G$. Recall the definitions of \emph{exceptional directions} for random forward trees (Definition~\ref{def-exceptionaldir}), and 
random backward trees (Definition~\ref{def-except}). 

The first (trivial) observation is that if $G$ is a free group with standard generators, then there are no exceptional directions since $\Gamma$ is a tree. Our main result in this section shows that if $G$ is not free (up to finite index) then exceptional directions exist and are dense in $\partial G$. 

\begin{defn}
A group $G$ is said to be \emph{virtually free} if there exists a free subgroup of finite index in $G$.
\end{defn}

We have the following theorem. 

\begin{theorem}\label{thm-nonfreedense}
	Let $G$ be a hyperbolic group such that $G$ is not virtually free, and let $\Gamma$ be its Cayley graph with respect to some finite symmetric generating set. Then the following hold.
	\begin{enumerate}
	\item For $\omega$ in the probability 1 event $\Omega_0$ from Theorem \ref{thm-12proper}, the set of exceptional directions in the forward tree $F(1,\omega)$ is non-empty and dense in $\pG$. 
	\item  For $\xi\in \pG$, and for $\omega$ in the probability 1 event $\Omega_0\cap \Omega_{\xi}$ from Theorem \ref{thm-12proper}, the set of exceptional directions in the backward tree $T(1,\omega)$ is non-empty and dense in $\pG$. 
\end{enumerate}	 
\end{theorem}
The reader will note that once we have fixed an $\omega$ as in the statement of the theorem, the rest of the argument is purely geometric and deterministic. {A word is in order about the case excluded from Theorem~\ref{thm-nonfreedense}, viz.\ that of virtually free groups.
We will pose this case explicitly  as Question~\ref{qn-other} and briefly discuss it there.  For now, we simply say that the issues involved are significantly different than those that go into Theorem~\ref{thm-nonfreedense}.}

The main tool for proving Theorem \ref{thm-nonfreedense} will be the following observation which identifies the exceptional directions with the multiple values of the Cannon-Thurston maps. 

\begin{rmk}\label{rmk-exception=ct}
Note that if $z$ is an exceptional direction for $F(1,\omega)$, there are distinct infinite rays in $F(1,\omega)$ which  accumulates on $z$ and a similar conclusion holds for $T(\xi,\omega)$. Therefore, by  Theorem~\ref{thm-12proper} and the continuity of Cannon-Thurston maps, for $\omega\in \Omega_0 \cap \Omega_{\xi}$ (resp.\ $\omega\in \Omega_0$), $z \in \pG$ is an exceptional direction for  $T(\xi,\omega)$ (resp. $F(1,\omega)$) if and only if $|(\partial i^{T}_\omega)^{-1}(z)|\geq 2$ (resp.\ $|(\partial i^{F}_{\omega})^{-1}(z)|\geq 2$) where $i^{T}_{\omega}$ (resp.\ $i^{F}_{\omega}$) is the Cannon-Thurston map constructed in Theorem \ref{thm-12proper} for $T(\xi, \omega)$ (resp.\ $F(1,\omega)$), i.e.\ $z$ is a multiple value of the Cannon-Thurston map $\partial i_\omega$.
Further, $|(\partial i^{T}_\omega)^{-1}(z)|$ equals the multiplicity (Definition~\ref{def-mult}) of the direction $z$.
\end{rmk}

Since the arguments in this section will be identical for both forward and backward trees, from now on we shall drop the superscripts $F$ and $T$. We shall assume henceforth that we are working with 
\begin{enumerate}
\item either a fixed $\omega\in \Omega_0$  for forward trees, 
\item  or a fixed $\xi \in \pG$ and a fixed $\omega\in \Omega_0\cap \Omega_{\xi}$ for backward trees.
\end{enumerate}
Further, $i_\omega$ will denote the inclusion of the corresponding random tree (forward or backward) into $\Gamma$. Also,  $\partial i_\omega$ will denote {the}  boundary value of the corresponding Cannon-Thurston map furnished by Theorem~\ref{thm-12proper}.

To prove Theorem~\ref{thm-nonfreedense}, we shall need a number of standard but deep facts about hyperbolic groups and their boundaries. For the convenience of the  reader we shall separate out the two parts of the proof. The main  input from geometric group theory needed is the following.

\begin{lemma}
\label{l:input}
Let $G$ be a  hyperbolic group that is not virtually free. Then there exists a quasiconvex hyperbolic subgroup $G'$ such that $\pG'$ is a compact connected and locally connected subset of $\partial G$. 
\end{lemma}

Postponing the proof of Lemma \ref{l:input}, we move on to the proof of Theorem \ref{thm-nonfreedense}. We shall need the following easy topological lemma. 

\begin{lemma}\label{lem-top}
Let $f: A \to B$ be a continuous surjective map from a totally disconnected set $A$ onto
a compact connected locally connected set $B$. Then multiple points of $B$ (i.e.\ $b \in B$ such that
$|f^{-1}(b)|>1$) is dense.
\end{lemma}

\begin{proof}
Suppose not. Then,  there exists $b \in B$ and a compact connected neighborhood $\overline{N_\ep (b)}$ of $b$ such that $f$ is injective and continuous on the compact set $f^{-1}(\overline{N_\ep (b)})$. Hence $f$ is a homeomorphism.
But $f^{-1}(\overline{N_\ep (b)})$ is totally disconnected,
and $\overline{N_\ep (b)}$ is connected, a contradiction.
\end{proof}

We can now complete the proof of Theorem \ref{thm-nonfreedense}. 

\begin{proof}[Proof of Theorem~\ref{thm-nonfreedense}] 
Let $G'$ be the subgroup given in Lemma \ref{l:input}. {(If the reader likes, (s)he may think of $G=G'*G{''}$ as a motivating example to keep in mind. Here $G'$ is assumed to be one-ended.)}
Then for each $g\in G$, $g\cdot\pG'$ is a connected locally connected compact subset of $\pG$. We first show that for each $g\in G$, the set of exceptional directions is dense in $g\cdot\pG'$. 

Indeed, consider the Cannon-Thurston map $\partial i_{\omega}:  \partial T(\omega) \to \partial G$ given by Theorem \ref{thm-12proper} (where $T(\omega)$ is either $F(1,\omega)$ or $T(\xi,\omega)$). Consider the restriction of this map to $\partial i_{\omega}^{-1}(g\cdot\partial G')$. By Theorem \ref{thm-12proper}, this map is continuous and surjective. Since $T(\omega)$ is a tree $\partial i_{\omega}^{-1}(g\cdot\partial G')$ is totally disconnected whereas $g\cdot\pG'$ is a connected locally connected compact as argued above. This therefore satisfies the hypothesis of Lemma \ref{lem-top} and hence the multiple points of this map are dense in $g\cdot\pG'$. By Remark \ref{rmk-exception=ct}, we get that the set of exceptional directions is dense in $g\cdot\pG'$, as claimed. 

So it remains to show that the set of exceptional directions is dense in all of $\partial G$. Towards this, it is a standard fact that for any $\xi' \in \pG$, $G\cdot\xi' $ is dense in $\pG$ (for any non-elementary hyperbolic $G$). Hence $\cup_{g\in G}\, (g\cdot\pG')$ is dense in $\pG$. We shall further show, using quasiconvexity of $G'$ that for any open set $U \subset \partial G$, there exists $ g \in G$ such that $g\cdot\partial G' \subset U$, which clearly shows that there are exceptional directions in $U$, completing the proof. 

Indeed, equip the space $C_c(\pG)$ of compact subsets of $\pG$ with the Hausdorff metric $d_\HH$, {i.e.\ $d_\HH(C_1, C_2)$ is given by the least $\ep$ such that the $\ep-$neighborhood of $C_1$ contains $C_2$ and vice versa}. Then  the only closed subsets of $\pG$ that are accumulated upon by elements of  $g\cdot\partial G'$, with $g \in G$ are singletons, i.e.\ points of $\pG$ . Further, each $\{\xi'\} \subset \pG$ is accumulated upon by a sequence $\{g_n \partial G'\}$ of translates of $\pG'$ (see \cite[Proposition 2.3]{mahan-relrig} for instance). Hence for $U$ as above containing $\xi$, there exists $ g \in G$ such that $g\cdot\partial G' \subset U$, as claimed.  
\end{proof}

We now move to  proving  Lemma \ref{l:input}. We need the following classical results. 

\begin{theorem}\cite{bowditch-cutpts,swarup-era}\label{thm-lc}
Let $G$ be a one-ended hyperbolic group. Then $\pG$ is connected and locally connected.
\end{theorem}

Stallings \cite{stallings-ends} in his foundational work shows that 
 a finitely generated group $G$ has more than one end  if and only if 
$G$ acts on a
simplicial tree $\TT$ such that
\begin{enumerate}
\item  for each edge $e \in \EE(\TT)$, the stabilizer $G_e$ is finite,
\item for each vertex $v \in \VV(\TT)$, the stabilizer $G_v$ is unequal to $G$.
\end{enumerate} 

We refer the reader to \cite{scott-wall} for a thorough topological treatment of Bass-Serre theory, in particular graphs of groups and trees of spaces.

\begin{defn}\label{def-accessible}
We say that a finitely generated group $G$ is \emph{accessible} if there exists a simplicial tree $\TT$ and a $G-$action on $\TT$ such that
 \begin{enumerate}
 \item $G_e$ is finite for each edge $e \in \EE(\TT)$,
 \item $G_v$ has at most one end for each vertex $v \in \VV(\TT)$.
 \end{enumerate}
\end{defn}

\begin{theorem}\cite{dunwoody} \label{dunwoody}
If $G$ is finitely presented, $G$ is accessible.
\end{theorem}

\begin{proof}[Proof of Lemma \ref{l:input}] Note that if $G' < G$ is quasiconvex, then the inclusion map naturally extends to an embedding
	of $\pG' < \pG$; this is the base case of a Cannon-Thurston map  \cite{mitra-trees} furnished by a quasi-isometric embedding. In particular, this also implies that the topology on $\partial G'$ obtained from the Gromov compactification of $G'$ is the same as the subspace topology inherited from $\pG$. 
By Theorem \ref{thm-lc} therefore, we only need to furnish a quasiconvex {one-ended} subgroup $G'$ of $G$. If $G$ is {one-ended} then we can take $G'=G$ and therefore, we can assume that $G$ has more than one end. Stallings' Theorem \cite{stallings-ends} gives a non-trivial graph of groups decomposition as pointed out following Theorem \ref{thm-lc}.
Since $G$ is hyperbolic, it is finitely presented \cite[Ch.4 Thm.1]{GhH}, and hence accessible by Theorem~\ref{dunwoody}. Let $\GG$ be the underlying graph of  a 
maximal graph of groups decomposition  of $G$ guaranteed  by Theorem~\ref{dunwoody} (we refer the reader to \cite{scott-wall} for the relevant Bass-Serre theory).

{
If all the vertex and edge groups are finite (i.e.\ have zero ends)
then $G$ is virtually free (since the resulting group is forced to have rational cohomological dimension one, and one can use classical results of Stallings and Swan \cite[ Ch. VIII, Example 2]{brown}).  Since $G$ is not virtually free, 
 at least one vertex group $G_v$  is one-ended. Choose $G'=G_v$.
 It suffices to prove that $G_v < G$ is quasiconvex. (Since a quasiconvex subgroup of a hyperbolic group is necessarily hyperbolic, Theorem~\ref{thm-lc} applies to $\pG'$.) Now, for a tree of hyperbolic spaces where the edge spaces are of uniformly bounded diameter, the
 resulting total space is hyperbolic \cite{BF} and
  each
 vertex space is quasiconvex (see  for instance the Bowditch path families argument in \cite[Section 3.1]{mahan-sardar}). It follows that $G'$ is quasiconvex as required.}
\end{proof}

\subsection{Cardinality  of exceptional directions}\label{sec-direxcard}
We next explore the cardinality of the set of exceptional directions. As mentioned before, in the planar exponential LPP case, the exceptional directions are known to be countable almost surely. It turns out that it is not so in the higher dimensional hyperbolic situation. 

\subsubsection{Uncountably many exceptional directions in higher dimensions}\label{sec-direxuncountable}
To state the result we shall need the notion of topological dimension of a metric space;  we refer the reader to \cite{engelking} for details on topological dimension. For a metric space $Z$, we shall denote its topological dimension by $\dim_t Z$. We shall recall below classical facts we need about $\dim_t Z$, but let us first state the main result.

\begin{theorem}\label{thm-uncountable}
Let $G$  be a hyperbolic group such that $\dim_t \pG > 1$. Then for $T(\omega)$ either the forward or backward tree {for any $\omega$ in the corresponding full measure subset as in the statement of Theorem \ref{thm-nonfreedense}}, the number of exceptional directions must be uncountable.
\end{theorem}

Before proving Theorem \ref{thm-uncountable} let us mention that the topological dimension of $\pG$ reflects cohomological properties of $G$ as the following fundamental theorem of Bestvina-Mess shows.
\begin{theorem}\cite[Corollary 1.4]{bes-mess}\label{thm-bm}
Let $G$ be a hyperbolic group. Then
\begin{enumerate}
\item $\dim_t \pG = \max\{n : H^n(G; \Z G) \neq {0}\},$
\item If $G$ is virtually torsion free (for instance if $G$ is residually finite) then  $\dim_t \pG =\mathrm{vcd}(G) - 1,$ 
where vcd refers to virtual cohomological dimension.
\end{enumerate}
\end{theorem}

Therefore the hypothesis $\dim_t \pG > 1$ can be replaced by any of the equivalent cohomological properties  in Theorem~\ref{thm-bm}.

Turning now to the proof of Theorem \ref{thm-uncountable}, we shall only need the following equivalent definition of topological dimension \cite[Definition 1.1]{topdim}. This is equal to the usual covering dimension for compact metric
spaces.

\begin{defn}\label{def-topdim}
Set $\dim_t \emptyset = -1$. The topological dimension of a non-empty metric
space $X $ is defined by induction as
$$\dim_t X=\inf\{d : X \ has\ a\ basis\ \UU 
\ such\ that\ \dim_t \partial U 
\leq d - 1 \ \ \forall  U \in \UU \}.$$
\end{defn}

\begin{proof}[Proof of Theorem \ref{thm-uncountable}:]
By Remark~\ref{rmk-exception=ct}, exceptional directions can be identified with multiple values of the Cannon-Thurston map
$\partial i_\omega: \partial T(\omega) \to \pG$. We argue by contradiction. Note that $ \partial T(\omega) , \pG$ are compact metrizable, and hence  separable (since this is true for any  proper Gromov-hyperbolic space). 

Suppose that the set
$A\subset \pG$ of exceptional directions is countable, then
\begin{enumerate}
\item $\dim_t A =0$, and
\item $(\pG \setminus A)$ is the image of the separable totally disconnected set
 $$Q=\partial T(\omega) \setminus (\partial i_\omega)^{-1} (A)$$ under the
 continuous  map $\partial i_\omega$, and the latter is injective on $Q$.
\end{enumerate}
{
Next, let $\rho$ denote a visual metric in $\pG$ and $S_\rho (x, \ep)$ denote the $\ep-$sphere about $x \in \pG$ with respect to $\rho$, i.e.\ 
$S_\rho (x, \ep)$ is the set of points in $\pG$ at $\rho-$distance equal to $\ep$ from $x$. 
Since
 $\dim_t \pG > 1$, we can choose $x \in \pG$ and $\ep > 0$ such that 
 \begin{enumerate}
 \item $S_\rho (x, \ep)$ is disjoint from the countable set $A$,
\item $\dim_t S_\rho (x, \ep) > 0$. 
 \end{enumerate}
 
 The second conclusion  follows from Definition~\ref{def-topdim}.  Indeed, the collection of $B_\rho (x, \ep)$ of open balls with boundaries $S_\rho (x, \ep)$ disjoint from $A$ forms a basis at the point $x$. Since $G\cdot x$ is dense in the compact metrizable $\pG$, the collection of translates $g\cdot B_\rho (x, \ep)$ forms a basis 
 of $\pG$. If each $S_\rho (x, \ep)$ has dimension zero, it follows from  Definition~\ref{def-topdim} that $\dim_t \pG\leq 1$, contradicting the hypothesis. 
 
  Since   $S_\rho (x, \ep)$ is compact metrizable, it has a connected component $K_\ep$ with 
 $\dim_t K_\ep > 0$ (since the maximum topological dimension of a connected component of a compact metric space equals the topological dimension of the ambient space). Being a component of the compact set $S_\rho (x, \ep)$,  $K_\ep$ is compact and  closed. Hence, $(\partial i_\omega)^{-1} (K_\ep)$ is closed in $\partial T(\omega)$, and hence compact.
 Since $\partial i_\omega$ is continuous and injective on $Q$, it must be a homeomorphism on $(\partial i_\omega)^{-1} (K_\ep)$. 
 This is a contradiction, since the image $K_\ep$ is connected. }
\end{proof}

\begin{example} We illustrate the proof of Theorem~\ref{thm-uncountable} by a simple example.
Let $G$ be the fundamental group of a closed hyperbolic 3-manifold. Then
$\pG$ is homeomorphic to $S^2$. Then $S^2\setminus A$ is path-connected for any countable $A$. In particular it contains a homeomorphic copy of $[0,1]$. The proof of Theorem~\ref{thm-uncountable} now shows that  the number of multiple points cannot be countable.
\end{example}

\subsubsection{Countably many exceptional directions for the hyperbolic plane}\label{sec-direxcountable} The cases not handled by Theorem~\ref{thm-uncountable} are hyperbolic groups $G$ with $\pG$ of dimension zero or one. If $\pG$ has dimension zero, then $G$ is virtually free. So we turn to hyperbolic groups $G$ with $\pG$ of dimension one.
We prove that for planar Cayley graphs in  the hyperbolic plane $\Hyp^2$,
the number of  exceptional directions must be countable.

\begin{theorem}\label{thm-countable}
	Let $G$  be a hyperbolic group acting cocompactly on the hyperbolic plane $\Hyp^2$ such that the Cayley graph $\Gamma$ is embedded in $\Hyp^2$. Then for $T(\omega)$ either the forward or backward tree, and {for any $\omega$ in the corresponding full measure subset as in the statement of Theorem \ref{thm-nonfreedense}}, the number of exceptional directions must be countable. 
\end{theorem}

\begin{proof}
Since $\Gamma$ is embedded in $\Hyp^2$, any random backward or forward  tree $T(\omega)$ is planar.
Further, the Cannon-Thurston map
 $\partial i_\omega$ from Theorem~\ref{thm-12proper}  is surjective.
The theorem now follows from the purely topological Lemma~\ref{lem-toptree} below and the fact that multiple values of the Cannon-Thurston map coincide with exceptional directions (Remark~\ref{rmk-exception=ct}). 
\end{proof}

The following is a purely topological lemma using the tautological circular order on the circle.

\begin{lemma}\label{lem-toptree}
Let $i:\TT \subset \Hyp^2$ denote an embedding of a planar tree.
Suppose that $i$ admits a surjective Cannon-Thurston map
$\partial i:\partial \TT \to \partial \Hyp^2$. Then there are countably many multiple values of $\partial i$.
\end{lemma}

\begin{proof}
Observe that $\partial \TT$ is totally disconnected and admits a circular order, i.e.\ there exists an \emph{embedding} $j: \partial \TT \to \CC$, where $\CC$ is an auxiliary circle, whose orientation agrees with the circular order on $\partial \TT$.
Note that the circular order on $\partial \TT$ is inherited from its embedding in $ \Hyp^2$, so that  $\CC$ and $\partial \Hyp^2$ are consistently oriented.
We identify $\partial \TT$ with its  image $j (\partial \TT) \subset \CC$.
Then $j (\partial \TT) $ is closed, and its complement is a countable collection of intervals $\{\II_n\}$ {since this is the structure of the complement of any closed subset in a circle}. 
Let $q: \CC\to \CC_0$ denote the quotient map that quotients each $\{\bbar \II_n\}$ down to a point $i_n \in \CC_0$.
Let $j_1=q \circ j$. Then
\begin{enumerate}
\item $\{i_n\}$ are precisely the multiple values of $q$, and hence of 
$j_1$. Further, each $i_n$ has precisely two pre-images under $j_1$
provided $\partial \TT$  is a Cantor set  (these are the `end-points' of the Cantor set). Else, all isolated points of  $\partial \TT$ are identified with end-points of neighboring intervals in $\CC$ under $j_1=q \circ j$.
\item Since the collection $\{\II_n\}$ is countable, $j_1$ has only countably many multiple values.
\item Since $q$ is a quotient map, the Cannon-Thurston map
$\partial i:\partial \TT \to \partial \Hyp^2$ factors as 
$\partial i = j_2\circ j_1$, where $j_2 : \CC_0 \to  \partial \Hyp^2$ is a surjective circular order-preserving continuous map from the circle 
$\CC_0 $ to the circle $\partial \Hyp^2$.
\end{enumerate}
Since $j_2$ is a surjective circular order-preserving continuous map, the pre-image of any point in $ \partial \Hyp^2$ is a closed interval, possibly a singleton. Thus, the multiple values of $j_2$ correspond precisely to the non-degenerate intervals (i.e.\ not singletons) in $\CC_0$ that are identified to points in $ \partial \Hyp^2$ under $j_2$.
Since any such interval {contains a rational number}, $j_2$ has at most countably many multiple values.

We have already seen that $j_1$ has only countably many multiple values. 
Hence $\partial i = j_2\circ j_1$ has only countably many multiple values, proving the lemma. 
\end{proof}

\subsection{Multiplicity of exceptional directions}\label{sec-mult}
In this subsection, we investigate the multiplicity (see Definition~\ref{def-mult}) of  random geodesics in an {exceptional}  direction $\xi$. This is the analog of the N3G problem in the planar Euclidean FPP/LPP setting. We show, in contrast to what is known {or} believed in that set-up, that the maximum number of distinct $\omega$-geodesics can be arbitrarily large depending on the topological dimension of $\pG$.  
The main result in this subsection is the following. 

\begin{theorem}\label{thm-mult}
Let $G$ be hyperbolic such that $\dim_t \pG = n-1$ {for some $n\ge 2$} (which implies that $G$ is not virtually free and Theorem \ref{thm-nonfreedense} applies). Then for $T(\omega)$ either the forward or backward tree, {and for any $\omega$ in the corresponding full measure subset as in the statement of Theorem \ref{thm-nonfreedense}}, there exists an exceptional direction $z \in \pG$ with multiplicity \emph{at least $n$}. 
\end{theorem}

The proof of this result hinges on the following nearly century-old theorem from topological dimension theory due to Hurewicz (1926) and Kuratowski (1932). Let {$\ind X$} denote the (small) inductive dimension of a topological space (see \cite[Chapter 1]{engelking}). We also implicitly use the basic theorem \cite[Theorem 1.7.7]{engelking} that equates 
$\ind X$ with the topological dimension for separable metric spaces $X$. We shall apply this to $\pG$ that is compact, metrizable, hence separable.
Since the following theorem is purely topological,  "Cantor set" in the statement below will refer to a metric space homeomorphic to the standard Cantor set.

\begin{theorem}\label{thm-hurewicz}\cite[p. 70, 1.7.D]{engelking}
A separable metric space $X$ satisfies the inequality $\ind X \leq  n$ (for some $n\geq 0)$ if and only if there exists a closed surjective mapping $f: Z \to X$ of a zero-dimensional (in the sense of topological dimension)  separable $Z$ with fibers of cardinality
at most $n + 1$.  

Further,
\begin{enumerate}
	\item For every compact metric space $ X$ with  $\dim_t X \leq  n$ (for some $n\geq 0)$  there exists a continuous mapping $f: Z \to X$ of a closed subspace $Z$ of the  Cantor set  onto the space $X$ with fibers of cardinality at most $n+ 1$.
	\item If for a compact  metric space $X$ there exists a continuous mapping $f: Z \to X$ of a closed subspace $Z$ of the Cantor set  onto the space X with fibers of cardinality at most $n+ 1$, then  $\dim_t X \leq  n$.
\end{enumerate} 
\end{theorem}

\begin{proof}[Proof of Theorem \ref{thm-mult}]
Let $\partial i_\omega$ denote the Cannon-Thurston map for the pair
$(T(\omega), \Gamma)$ furnished by Theorem~\ref{thm-12proper}.
Then $\partial T(\omega)$ is (homeomorphic to) a closed subset of a Cantor set {(since this is true for the boundary of any tree with bounded valence).}
Also, $\partial i_\omega: \partial T(\xi,\omega) \to \pG$ is surjective.
Hence by  Theorem~\ref{thm-hurewicz}(2), there exists $z \in \pG$ such that
$\partial i_\omega^{-1}(z)$ has cardinality at least $n$. 
Finally,
by Remark~\ref{rmk-exception=ct}, $z$ has multiplicity at least  $n$. 
\end{proof}

\begin{example}\label{eg-ndim}
For completeness, we furnish examples in each dimension. Let $G$ be a cocompact lattice in $Isom^+(\Hyp^{n+1}) = SO(n+1,1)$, i.e.\ $G$ is the fundamental group of a closed hyperbolic $(n+1)-$manifold
(here $Isom^+(X)$ denotes the group of orientation-preserving isometries of $X$). Then $G$ is quasi-isometric to 
the hyperbolic $(n+1)-$space $\Hyp^{n+1}$ by the Milnor-Schwarz lemma and hence $\pG = \partial \Hyp^{n+1}$ is homeomorphic to the $n-$sphere $S^n$. These furnish examples to which
Theorem~\ref{thm-mult} applies. 
\end{example}

\section{Exceptional directions: upper bounds}
\label{sec-exceptional-coalesce}
The aim of this section is to prove the upper bounds on cardinality and multiplicities of exceptional directions, i.e., Theorem \ref{thm-omni-intro}, (5) and (6). To reduce the notational overhead, {the proofs in this section will be written for the forward geodesic trees $F(1,\omega)$. We shall explain subsequently how the same arguments work for the backward trees $T(\xi,\omega)$ as well; see Remarks \ref{r:btree1} and \ref{r:btree2}.
 Recall that for a full measure set $\Omega'$ and for $\omega\in \Omega'$, the union of all $\omega$-geodesics started at $1$ form the forward tree $F(1,\omega)$.   
 The Patterson-Sullivan measure of $K \subset \pG$ will be denoted by $\nu(K)$.}

\subsection{Exceptional directions are exceptional}
We call two $\omega$-geodesic rays distinct if they are disjoint outside some finite ball. A point $\xi\in \pG$ is said to be a $1$-exceptional direction for the forward geodesic tree $F(1,\omega)$ if there exist at least two distinct geodesic rays in $F(1,\omega)$ converging to $\xi$. {Note that this is equivalent to Definition~\ref{def-exceptionaldir}.}
Let $A(\omega)\subset \pG$ denote the set of all $\omega$-exceptional directions. Our goal in this subsection is to prove the following theorem. 

\begin{theorem}
	\label{t:exceptional}
	For $\P$-a.e. $\omega$, $A(\omega)$ has zero Patterson-Sullivan measure, i.e., $\nu(A(\omega))=0$. 
\end{theorem}

Theorem \ref{t:exceptional} will follow from Proposition~\ref{p:exceptional} below, which furnishes a stronger result. For $v\in \Gamma$, let $\widetilde{A}_v(\omega)$ denote the set of all directions in $\pG$ such that there exists two edge disjoint geodesic rays  from $v$ in the direction $\xi$. In this section, we assume that
$\pG$ is equipped with the \emph{visual metric}, so that it becomes a metric measure space once it is equipped with the 
Patterson-Sullivan measure. Thus, $\pG$ is \emph{Ahlfors-regular}, i.e.\ the volume of balls of radius $r$ in $\pG$ have volume $\sim r^\DD$ \cite{coornert-pjm}, where $\DD$ is the Hausdorff dimension of 
$\pG$ equipped with the visual metric and the 
Patterson-Sullivan measure. 
We shall prove the following.

\begin{prop}
	\label{p:exceptional}
	For $\P$-a.e. $\omega$, $\widetilde{A}_1(\omega)$ has
	Hausdorff dimension strictly less than $\pG$ equipped with the 
	Patterson-Sullivan measure and the visual metric. In particular, $\widetilde{A}_1(\omega)$ has
	zero Patterson-Sullivan measure, i.e., $\nu(\widetilde{A}_{1}(\omega))=0$.
\end{prop}

The last statement of Proposition~\ref{p:exceptional} follows from the first since $\pG$  equipped with the Patterson-Sullivan measure and the visual metric  is Ahlfors-regular
\cite[Theorem 7.2, Proposition 7.4]{coornert-pjm}. It thus suffices to prove the first.
Observe that, by almost sure uniqueness of geodesics between any $u,v\in \Gamma$, it follows that $A(\omega)\subseteq \cup_{v} \widetilde{A}_{v}(\omega)$. Proposition \ref{p:exceptional} implies by group invariance that for all $v\in \Gamma$, for $\P$-a.e.\ $\omega$, we have $\nu(\widetilde{A}_v(\omega))=0$. By taking a countable union over all $v$ (and a countable intersection over the corresponding probability 1 events), we get Theorem \ref{t:exceptional}. In fact, we get the stronger assertion that the Hausdorff dimension  of $A(\omega)$ is strictly less than that of $\pG$.

{\begin{rmk}
\label{r:measure}
Notice that if one just wants to show that $\nu(A(\omega))=0$ for a.e.\ $\omega$, one might want to use Theorem \ref{bmthmcoalesce} which says that $\P(\xi\in A(\omega))=0$ for each $\xi\in \pG$ and use Fubini's theorem. Indeed, if one could show that the random variable $I(\xi\in A(\omega))$ is jointly measurable in $\xi$ and $\omega$, then Fubini's theorem and Theorem \ref{bmthmcoalesce} would indeed imply that $\nu(A(\omega))=0$ for a.e. $\omega$. This question of measurability, however, does not seem entirely straightforward. Hence we provide the argument below which leads to a stronger result in terms of the Hausdorff dimension. 
\end{rmk}
}

To prove Proposition \ref{p:exceptional}, we need the following. For $n\in \N$ large, and  $v\in \partial B_n(1)$, the $n$-sphere, we say that $v$ is a \emph{lift of an exceptional direction} $\xi\in \pG$ if $\xi\in \widetilde{A}_1(\omega)$ and $v\in [1,\xi)$. Let $C_{n}\subset \partial B_n(1)$ denote the set of lifts of exceptional directions. We shall prove the following lemma. 

\begin{lemma}
	\label{l:exceptionallift}
	There exists $c_1,c_2>0$ such that for all $n$ sufficiently large we have 
	$$\P\bigg(|C_n|\ge \exp(-c_1n)|\partial B_n(1)|\bigg) \le \exp(-c_2n) $$
\end{lemma} 

Proposition \ref{p:exceptional} now follows from Lemma \ref{l:exceptionallift} by a standard application of Sullivan's shadow lemma \cite[Section 6]{coornert-pjm} as follows. What the argument in 
\cite[Section 6]{coornert-pjm} shows is that the (upper) volume entropy 
$\limsup \frac{\log |C_n|}{n}$ is at least the Hausdorff dimension of the 
exceptional set $\widetilde{A}_1$. Indeed, for a sufficiently large, fixed $R$
depending on $\delta$ alone, we have the following.
\begin{enumerate}
	\item $\limsup \frac{\log |C_n|}{n}$ equals $\limsup \frac{\log |B_R(C_n)|}{n}$, where $B_R(C_n)$ denotes the $R-$neighborhood of $C_n$.
	\item For all $n$ large enough, the $R-$shadow of $C_n$ covers {$\widetilde{A}_1(\omega)$} by Lemma~\ref{l:wandering0} and ~\ref{l:wandering}. (Here, by $R-$shadow of a set $W$ we mean the collection of points $z \in \pG$, such that $[1,z)$ passes through an $R-$neighborhood of $W$.)
	\item Finally, by the Borel-Cantelli Lemma, and Lemma \ref{l:exceptionallift}, we have a.s.\ 
	 $$\limsup \frac{\log |C_n|}{n}\le \lambda-c_1,$$
	 where $c_1$ is as in the statement of Lemma \ref{l:exceptionallift}, 
	 and $\lambda = \lim_{n \to \infty} \frac{\log |\partial B_n(1)|}{n}$
	 equals the volume entropy of $\Gamma$.  Note that $\lambda$ equals the Hausdorff dimension of
	 $\pG$ equipped with the visual metric and the Patterson-Sullivan measure \cite{coornert-pjm}.
\end{enumerate}

Thus, it suffices to 
prove Lemma \ref{l:exceptionallift}. We consider two different events. For $\varepsilon>0$, let us call a  semi-infinite $\omega$-geodesic ray $\gamma$ (starting at 1) $\varepsilon$-wandering at distance $n$ if it does not intersect $B_{\varepsilon n}(v)$ where $v=[1,\xi)\cap \partial B_n(1)$ and $\xi$ is the asymptotic direction of $\gamma$. Let $\mathscr{C}_{\varepsilon,n}$ denote the event that there does not exist any $\omega$-geodesic ray started at $1$ that is  $\varepsilon$-wandering at distance $n$. Further for $v\in \partial B_n(1)$, let $X_{v}$ denote the indicator of the event that there exists $v_1,v_2\in B_{\varepsilon n}(v)$ such that the geodesics $[1,v_1]_{\omega}$ and $[1,v_2]_{\omega}$ are edge disjoint. Clearly, on the event $\mathscr{C}_{\varepsilon,n}$ we have 
$$ |C_n| \le \sum_{v\in \partial B_{1}(n)}X_{v}.$$
Therefore, Lemma \ref{l:exceptionallift} will follow from  Lemmas~\ref{l:wandering0} and ~\ref{l:wandering} below. 

\begin{lemma}
	\label{l:wandering0}
	For each $\varepsilon>0$, there exists $c=c(\varepsilon)>0$ such that for all $n$ sufficiently large 
	$$\P(\mathscr{C}_{\varepsilon,n})\ge 1-\exp(-cn).$$
\end{lemma}

\begin{proof}
	To prove Lemma~\ref{l:wandering0} we start with some hyperbolic geometry observations: \\
	Let $\pi_\xi$ denote nearest-point-projection onto $[1,\xi)$ (a deterministic
	geodesic ray from 1 to $\xi$). 
	We then note that  $[1, \xi)_\omega$ is a random geodesic with direction $\xi$ if and only if $\pi_\xi$ maps  $[1, \xi)_\omega$  coarsely 
	surjectively onto $[1, \xi)$. Next, let
	$\YY_n (\xi,\omega)$ be the set of  points in  $ [1, \xi)_\omega$ whose image under $\pi_\xi$  is  at distance $n$ from 1 on $[1, \xi)$.
	For convenience of notation, let us denote the point on $[1, \xi)$ at distance $n$ from 1  by $n(\xi)$.
	Then for any $y_{2n}(\xi,\omega) \in {\YY_{2n}} (\xi,\omega)$,   $[1,2n(\xi)]\cup[2n(\xi), y_{2n}]$ is a uniform quasigeodesic (in fact, $[1,2n(\xi)]\cup[2n(\xi), y_{2n}(\xi,\omega)]$ lies within
	a $2\delta-$neighborhood of the geodesic $[1,y_{2n}(\xi,\omega)]$). Further, there exists $\kappa$ depending only on $\delta$ such that  nearest-point-projections of any $y_{n}(\xi,\omega) \in {\YY_{n}} (\xi,\omega)$ onto  $[1,y_{2n}(\xi,\omega)]$ lies within a $\kappa-$distance of $n(\xi)$ (since
$[1,y_{2n}(\xi,\omega)]$ and $[1,2n(\xi)]\cup[2n(\xi), y_{2n}(\xi,\omega)]$ track each other uniformly).

Recall the events $\AAA_v=\AAA_v(C,\epsilon)$ and $\AAA_n=\AAA_{n}(C,\epsilon)$ from Lemma~\ref{lem-intnAn}. Observe from the proof of Lemma \ref{lem-intnAn} that using Lemma \ref{lem-effbt} and taking a union bound over $v\in \partial B_n(1)$ we have that for all $\varepsilon>0$, there exists $c>0$ such that 
$$\P(\AAA_n(\kappa,\varepsilon))\ge 1-\exp(-cn)$$ for all $n$ sufficiently large. It remains to verify that $\AAA_n(\kappa,\varepsilon) \subseteq \mathscr{C}_{\varepsilon,n}$. Indeed, note that for every $v\in \partial B_n(1)$ and for every $\xi\in \partial G$ such that $n(\xi)=v$, we have that $[1,y_{2n}(\xi,\omega)]$ passes within $\kappa$ distance of $v$ (see the geometric argument in the previous paragraph). Therefore, on the event $\AAA_v(\kappa,\varepsilon)$ we have that for all $\xi\in \pG$ such that $n(\xi)=v$ there does not exist any $\omega$ geodesic ray with direction $\xi$ that is $\varepsilon$-wandering at distance $n$. Taking a union bound over all $v\in \partial B_n(1)$ (and using the fact that $n(\xi)\in \partial B_n(1)$ for all $\xi\in \pG$), we get $\AAA_n(\kappa,\varepsilon) \subseteq \mathscr{C}_{\varepsilon,n}$, as claimed. This completes the proof.  
\end{proof}

\begin{lemma}
	\label{l:wandering}
	Let $\varepsilon>0$ be fixed and sufficiently small. There exists $c'=c(\varepsilon)>0$ and $\tilde{c}=\tilde{c}(\varepsilon)>0$ such that for all $n$ sufficiently large 
	$$\P\left(\sum_{v\in \partial B_{1}(n)}X_{v} \ge \exp(-\tilde{c}n)|\partial B_n(1)|\right)\le \exp(-c'n).$$
\end{lemma}

\begin{figure}[htbp!]
\includegraphics[width=14cm]{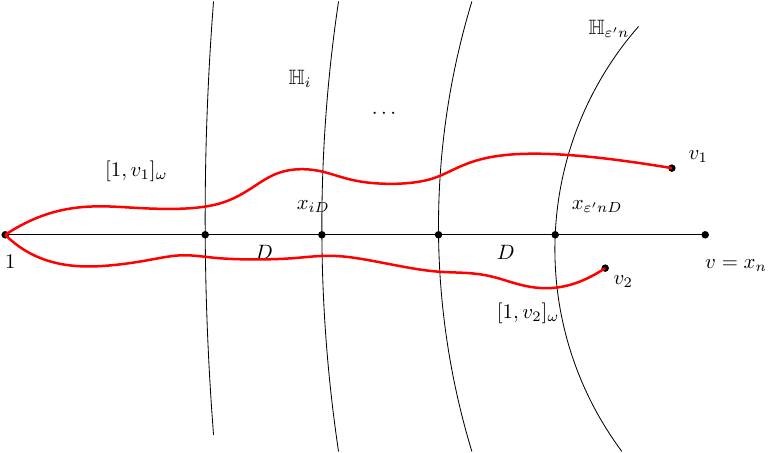}
\caption{Proof of \eqref{e:xvbound}: We consider the hyperplanes $\mathbb{H}_{i}$ perpendicular to $[1,v]$ with separation $D$ between consecutive hyperplanes. If there exist points $v_1,v_2$ within distance $\varepsilon n$ of $v$ such that $[1,v_1]_{\omega}$ and $[1,v_2]_{\omega}$ are edge disjoint then there must be disjoint geodesics from $\mathbb{H}_i$ to $\mathbb{H}_{i+1}$ for all $i\le \varepsilon' n$. Since the probability of this event is bounded away from 1, and the events are independent for all even $i$, we get the probability that there exist such disjoint geodesics is exponentially small.} 
\label{f:hyperplane}
\end{figure}

\begin{proof} The first step is to prove that there exists $c_*>0$ such that for all $v\in \partial B_n(1)$
	\begin{equation}
		\label{e:xvbound}
		\P(X_v=1)\le \exp(-c_*n).
	\end{equation}
	To see this, fix $v$ and consider $[1,v]={1=x_0,x_1,\ldots, x_{n}=v}$. For $D$ sufficiently large, let $\mathbb{H}_{i}$ denote the hyperplane passing through $x_{iD}$ perpendicular to $[1,v]$.
	(In \cite{BM19}, we constructed hyperplanes perpendicular to
	infinite geodesic rays, but  the same construction works for finite geodesic segments as well.) We know from \cite[Proposition 7.10]{BM19} that for $D$ sufficiently large, there is $p>0$ (not depending on $v$ or $n$)  and events $\mathscr{H}_{i}$ such that 
	
	\begin{itemize}
		\item[(i)] $\P(\mathscr{H}_{i})\ge p$.  
		\item[(ii)] The events $\mathscr{H}_{2i}$ are independent.
		\item[(iii)] For all $i\ge i_0$, on the event $\mathscr{H}_{i}$ every pair of geodesics from a point in $\mathbb{H}_{i}$ to a point in $\mathbb{H}_{i+1}$ share a common edge. 
	\end{itemize}
{(That $p$ does not depend on $v$ is a consequence of the proof of \cite[Proposition 7.10]{BM19}.)}
	
	Observe also that for every $v'\in B_{\varepsilon n}(v)$, $[1,v']_{\omega}$ must cross from $\mathbb{H}_i$ to $\mathbb{H}_{i+1}$ for each $i\le \varepsilon' n$ where $\varepsilon'=\varepsilon'(\varepsilon,D)>0$ is independent of $v$ and $n$.
	{This is because $B_{\varepsilon n}(v)$ is $\delta-$quasiconvex, and each $\mathbb{H}_i$ has quasiconvexity constant dependent only on $\delta$. Further, each $\mathbb{H}_i$ separates $\Gamma$}. 
	
	It is therefore clear (due to the almost sure uniqueness of $\omega$-geodesics between any two points) that on the event $\cup_{i\le \varepsilon' n} \mathscr{H}_i$ we have 
	$X_v=0$ almost surely. Since the events $\mathscr{H}_{2i}$ are independent we get that 
	$$\P(X_{v}=1)\le (1-p)^{\varepsilon' n/2}$$
	completing the proof of \eqref{e:xvbound}. 
	
	Next, \eqref{e:xvbound} implies that 
	$$\E\left[\sum_{v\in \partial B_{1}(n)}X_{v}\right] \le \exp(-c_*n)||\partial B_n(1)|.$$
	Choosing $\tilde{c}\in (0,c_*)$ Markov inequality implies {
		$$\P\left(\sum_{v\in \partial B_{1}(n)}X_{v} \ge \exp(-\tilde{c}n)|\partial B_n(1)|\right)\le \exp(-(c_{*}-\tilde{c})n).$$
		This completes the proof of the lemma.}
\end{proof}

\begin{rmk}
\label{r:btree1}
We briefly describe how to get Theorem \ref{t:exceptional} and Proposition \ref{p:exceptional} for the backward tree $T(\xi,\omega)$. Observe that if $\xi'\ne \xi$ is an exceptional direction in $T(\xi,\omega)$ for $\omega\in \Omega_{\xi}$, there exists $v\in \Gamma$ and disjoint rays $\gamma_1$, $\gamma_2$ from $v$ in direction $\xi'$ such that $\gamma_1\cup [v,\xi)_{\omega}$ and $\gamma_2 \cup \cup [v,\xi)_{\omega}$ are both bigeodesics $(\xi',\xi)_{\omega}$ contained in $T(\xi,\omega)$. In particular, $\gamma_1$ and $\gamma_2$ are both semi-infinite geodesics $[v,\xi')_{\omega}$. Therefore, $\xi'\in \widetilde{A}_{v}(\omega)$. Therefore, the set of exceptional directions for $T(\xi,\omega)$ is also contained in $\cup_v \widetilde{A}_{v}(\omega)$. The result now follows from Proposition \ref{p:exceptional}. 
\end{rmk}

\subsection{Maximum multiplicity of exceptional directions is bounded}
We have shown in Theorem~\ref{thm-uncountable}  that for groups which are not virtually free, there exist exceptional directions $\xi$ for the forward tree such that there are at least two distinct $\omega$-geodesics from  $1$ in the direction $\xi$. We also know from Theorem~\ref{thm-mult} that the maximum number of distinct geodesics in exceptional directions can be arbitrarily large depending on the group. Here we show that for any fixed Cayley graph $\Gamma$ there is a uniform upper bound to {the} number of distinct geodesics in any direction for a full measure subset of $\Omega$. 

\begin{theorem}
	\label{t:finite}
	There exists $k_0=k_0(\Gamma,\delta)$ such that almost surely there does not exist any direction $\xi\in \pG$ with multiplicity of $\xi$  more than $k_0$, i.e.\ 
	there do not exist more than $k_0$  distinct $\omega$-geodesics from $1$ in the direction $\xi$.  
\end{theorem}

This result is along the same lines as the no three disjoint geodesics results in planar  exponential LPP model. Although we do not exactly know what is the maximum number of geodesics in an exceptional direction, Theorem~\ref{t:finite}  shows that there cannot be too many.

The proof of Theorem \ref{t:finite} uses arguments broadly similar to the proof of Theorem \ref{t:exceptional}. The result will follow from the following estimate together with the Borel-Cantelli lemma. 

\begin{proposition}
	\label{p:finite}
	Let $A_{n,k}$ denote the event that there exists $\xi\in \partial G$ such that there exist $k$ $\omega$-geodesic rays from $1$ in the direction $\xi$ which are disjoint outside $B_n(1)$. Then there exists $k_0$ sufficiently large such that for $k\ge k_0$, there exists $c>0$ such that 
	$$\P(A_{n,k})\le \exp(-cn).$$
\end{proposition}

For the proof of this proposition we shall again use the hyperplane construction. For $v\in \Gamma$ with $d(1,v)=2n$, let $v'$ denote the point on $[1,v]$ at a distance $n$ from $1$, and let $v''$ denote the point on $[1,v]$ at distance $3n/2$ from $1$. Consider hyperplanes $\mathbb{H}_{v'}$ and $\mathbb{H}_{v''}$ perpendicular to $[1,v]$ at points $v'$ and $v''$ respectively. Let $A_{v,k}$ denote the event that there exist $k$ disjoint $\omega$-geodesics from $\mathbb{H}_{v'}$ to $\mathbb{H}_{v''}$. It follows that $\cup_{v\in \partial B_1(2n)} A_{v,k} \supseteq A_{n,k}$. Since $|\partial B_1(2n)|\le \exp(Cn)$ for some $C>0$, Proposition \ref{p:finite} follows from the following lemma together with a union bound. 

\begin{lemma}
	\label{l:finite}
	For any large but finite $c$, there exists $k_0=k_0(c)$ such that for all $k\ge k_0$ and all $v\in \partial B_1(2n)$, we have 
	$$\P(A_{v,k})\le \exp(-cn).$$ 
\end{lemma}

\begin{proof}
	Let us fix $v\in \partial B_1(2n)$. As in the proof of Lemma \ref{l:wandering}, consider perpendicular hyperplanes $\mathbb{H}_i$ to $[1,v]$. Assume without loss of generality that $n/2$ is an integer multiple of $D$. Let $B_{i,k}$ denote the event that there exists $k$ disjoint $\omega$-geodesics from $\mathbb{H}_{i}$ to $\mathbb{H}_{i+1}$. Clearly, 
	
	$$ A_{v,k} \subseteq \cap_{n/D \le i \le 3n/2D} B_{i,k}.$$ 

Next we shall replace the events $B_{i,k}$ by slightly larger events. Note that for any $\omega$-geodesic $\gamma$ from $x_i\in \mathbb{H}_i$ to $x_{i+1}\in \mathbb{H}_{i+1}$, there exists a subsegment $\gamma_1$ from $x'_{i}$ to $x'_{i+1}$ such that $\gamma_1$ is also a \emph{slab geodesic} between $x'_{i}$ and $x'_{i+1}$, i.e., $\gamma_1$ minimizes the $d_{\omega}$-length among all paths from $x'_{i}$ to $x'_{i+1}$
contained in the region (or slab) $\SSS_i$ between $ \mathbb{H}_i$ and $ \mathbb{H}_{i+1}$. 

Notice now that (see \cite[Lemma 7.3]{BM19}) the slabs $\SSS_i$ are uniformly quasiconvex with quasiconvexity constant dependent only on $\delta$. Let $\widetilde{\SSS}_{i}$ denote a uniform thickening of $\SSS_{i}$ such that $\widetilde{\SSS}_{i}$ is quasi-isometrically embedded in $\Gamma$, and hence are $\delta'$-hyperbolic with $\delta'$ independent of $i$. Let $\widetilde{B}_{i,k}$ denote the event that there exists $k$ disjoint paths from points in $ \mathbb{H}_i$ to points in $ \mathbb{H}_{i+1}$ which are $\omega$-geodesics in the slab $\widetilde{\SSS}_i$. By the above discussions, $B_{i,k}\subseteq \widetilde{B}_{i,k}$. 

The reason we move from $B_{i,k}$ to $\widetilde{B}_{i,k}$ is that $\widetilde{B}_{i,k}$ depends only on the edge weights restricted to $\widetilde{\SSS}_{i}$. If $D>D_0$ for some large but fixed $D_0$ then the slabs $\widetilde{\SSS}_{2i}$ are disjoint and hence the events $\widetilde{B}_{2i}$ are independent. 

\noindent
\textbf{Claim.} For any $\varepsilon>0$ there exists $k=k(\varepsilon)$ sufficiently large (not depending on $v$) such that $\P(B_{i,k})\le \varepsilon$.
	
	Let us first complete the proof of the lemma assuming the claim. Using the independence of $\widetilde{B}_{2i,k}$
	$$\P(A_{v,k})\le \prod_{n/2D \le i \le 3n/4D} \P(\widetilde{B}_{2i,k})\le    (\varepsilon)^{n/4D}\le \exp(-cn)$$
	by choosing $\varepsilon$ sufficiently small depending on $c$ and $D$.

	Moving now to the proof of the claim,  $\omega$-geodesics in slabs $\widetilde{\SSS}_i$ between 
		$x_{i}\in \mathbb{H}_i, x_{i+1}\in \mathbb{H}_{i+1}$ will be denoted  as $[x_i,x_{i+1}]_{\SSS,\omega}$. Let $v^*$ denote the midpoint of $[v_i,v_{i+1}]$. Let $C_{i,r}$ denotes the event that there exists $x_{i}\in \mathbb{H}_i, x_{i+1}\in \mathbb{H}_{i+1}$ such that
		 $[x_i,x_{i+1}]_{\SSS,\omega}\cap B_{r}(v^{*})=\emptyset$. Let $k=k(r)$ denote the number of edges contained in $B_{r}(v_*)$ (this does not depend on $v_*$). It follows (since a family of edge disjoint geodesics can use at most one edge in $B_r(v^*)$) that $C_{i,r}\supseteq \widetilde{B}_{i,k+1}$. To prove the claim, it therefore suffices to show that for any $\varepsilon>0$, one can find $r$ sufficiently large such that $\P(C_{i,r})\le \varepsilon$.

		Let $y_i \in \mathbb{H}_i$ and $y_{i+`1} \in \mathbb{H}_{i+`1}$ be arbitrary points. Then, there exists $\kappa, D_0 >0$ such that if
		the above separation distance $D$ is at least $D_0$, and  $w \in [v_i,v_{i+`1}]$ is arbitrary, then the word geodesic $[y_i,y_{i+`1}]$ passes through a $\kappa$ neighborhood of $w$. Further, by uniform quasiconvexity of the slabs $\widetilde{\SSS}_{i}$, the same is true for geodesics in the intrinsic metric on $\til\SSS_i$. Further, the latter tracks $[y_i,y_{i+`1}]$.  We shall therefore ignore this difference and denote the geodesics in  $\til\SSS_i$ as $[y_i,y_{i+`1}]$. In particular, we have $B_{\kappa}(v_*)\cap [y_{i},y_{i+1}]\ne 0$ for all $y_i \in \mathbb{H}_i$ and $y_{i+`1} \in \mathbb{H}_{i+`1}$. Since $\widetilde{\SSS}_i$ is a subgraph of $\Gamma$, it has the degree upper bounded by the valence of $\Gamma$, and is $\delta'$-hyperbolic for some $\delta'$ dependent only on $\delta$. This is true
		for all $i$. It follows from Corollary \ref{cor-effbt} that there exists $c>0$ such that for all $i$ and for all sufficiently large $r$, 
		$\P(C_{i,r})\le \exp(-cr)$. By taking $r$ sufficiently large depending on $\varepsilon$, we get $\P(C_{i,r})\le \varepsilon$, as required. 
		This completes the proof of the claim and hence that of the lemma. 
	\end{proof}

\begin{rmk}
\label{r:btree2}
{We briefly sketch how to get the upper bound on maximum multiplicity of exceptional directions for the backward geodesic tree $T(\xi,\omega)$. Since $T(\xi,\omega)$ almost surely is the union of unique semi-infinite geodesics geodesics $[v,\xi)_{\omega}$ that coalesce, if the multiplicity of some exceptional direction $\xi'\neq \xi$ is at least $k$, it follows that there exists $M$ and $w\in B_{M}(1)$ such that there exists $k$ semi-infinite geodesic rays in the direction $\xi'$ starting from $w$ such that they are disjoint outside $B_{n}(1)$ for all $n$ large. Now fix $M$, $w\in B_{M}(1)$ and choose $n\gg M$. The hyperplane argument from Proposition \ref{p:finite} shows that for $k$ large (not depending on $M$) the probability that there exists $k$ semi-infinite geodesics from $w$ to some $\xi'$ which are disjoint outside $B_{n}(1)$ is bounded by $e^{-cn}$ for some $c>0$. Since $\sum_n e^{-cn} < \infty$, the first Borel-Cantelli Lemma now implies that the above event almost surely cannot hold for infinitely many $n$ and hence with probability $1$ there cannot exist $k$ semi-infinite geodesic rays in some direction $\xi'$ starting from $w$ such that they are disjoint outside $B_{n}(1)$ for all $n$ large. The claimed upper bound of $k$ on the  multiplicity of exceptional directions follows: first take a union bound over finitely many $w\in B_{M}(1)$ to get the conclusion almost surely simultaneously for all $w\in B_M(1)$ and then take a limit as the radius $M\to \infty$.} 
\end{rmk}

{
\begin{rmk}\label{rmk-topdim}
Theorem~\ref{t:finite} establishes an upper bound on the multiplicity in any direction $\xi$. The techniques flow from the proof of Theorem~\ref{t:exceptional}, where Hausdorff dimension bounds are obtained.

Theorem~\ref{thm-mult} on the other hand, establishes a lower bound in terms of the topological dimension of $\pG$ on the multiplicity in some direction $\xi$. 

It is a standard fact that the Hausdorff dimension dominates topological dimension. It would thus be worth establishing that the  upper bound  of
Theorem~\ref{t:finite} can in fact be refined to give a bound in terms of the   Hausdorff dimension.
\end{rmk}}

\begin{rmk}
\label{r:altdef2}
We note that the proofs in this section work under a weaker definition of exceptional directions. For $u,u'\in \Gamma$, let $A_{u,u'}(\omega)$ denote the set of directions $z\in \pG$ such that there exist $\omega$ geodesic rays from $u$ and $v$ in the direction $z$ that do not coalesce. Observe that for a.e. $\omega$, $A_{u,u'}(\omega)\subseteq \cup_{w} \widetilde{A}_{w}(\omega) \cup \widetilde{A}_{u,u'}(\omega)$ where $\widetilde{A}_{u,u'}(\omega)$ is the set of directions $z\in \pG$ such that there exist disjoint $\omega$ geodesic rays from $u$ and $v$ in the direction $z$. Observe now that if $n\gg d(1,u)+d(1,u')$, the same argument as in the proof of Lemma \ref{l:wandering} shows that for all $v\in \partial B_n(1)$, the probability that there exists $v_1,v_2\in B_{\varepsilon n}(v)$ such that $[u,v_1]_{\omega}$ and $[u',v_2]_{\omega}$ are disjoint has probability exponentially small in $n$. The same argument as in the proof of Proposition \ref{p:exceptional} will then lead to an upper bound on the Hausdorff dimension of $A_{u,u'}(\omega)$. By taking a union bound over countably many pairs $(u,u')$, Theorem \ref{t:exceptional} also holds for the weaker definition of exceptional direction as in Remark \ref{r:altdef}. Similarly, as already remarked in Remark \ref{r:btree2}, the bound in Proposition \ref{p:finite} also works even when the geodesics are allowed to start from different points.
\end{rmk}

\section{Questions and Generalizations}\label{sec-qns} In this section, we collect together a number of questions that the present paper gives rise to.

\subsection{One and two-dimensional  cases}\label{sec-2d}
In this subsection, we elaborate on the cases not covered by Section~\ref{sec-direxcard}.
For any hyperbolic group $G$,
Stallings' Theorem on ends of groups \cite{stallings-ends} coupled with
Dunwoody's accessibility Theorem~\ref{dunwoody} gives a  maximal unique graph of groups decomposition with finite edge groups and vertex groups that are either
finite or have one end. This is a standard first reduction step allowing us to focus on one-ended hyperbolic groups.
A next step involves decomposing one-ended hyperbolic groups along 2-ended subgroups, i.e.\ along subgroups that are virtually $\Z$.
The canonical JSJ decomposition of Rips-Sela-Bowditch \cite{rips-sela,rips-sela-jsj,bowditch-cutpts} furnishes us with precisely such a graph of groups decomposition where each edge group is 2-ended. By convention, surface groups are regarded as vertex groups for the JSJ decomposition; in particular,
surface groups have \emph{trivial JSJ decomposition} though they do split over infinite cyclic groups.
The details are not relevant to the present paper. However, we need the following theorem for our discussion:

\begin{theorem}\cite{kapovich-kleiner}\label{thm-kk}
	Let $G$ be a one-ended torsion-free hyperbolic group with 1 dimensional boundary, and trivial JSJ decomposition. Then $\pG$ is exactly one of the following:
	\begin{enumerate}
		\item a circle $S^1$,
		\item  a Sierpinski carpet
		\item  a Menger curve.
	\end{enumerate}
	If $G$ does not split over $\Z$, $\pG$ is either  a Sierpinski carpet or  a Menger curve.
\end{theorem}

Theorem~\ref{thm-uncountable} deals with $G$  having $\dim_t \pG > 1$. This reduces the problem of exceptional directions to that of 
$G$  having $\dim_t \pG = 1$.  Theorem~\ref{thm-countable} essentially deals with the first case of Theorem~\ref{thm-kk}.
{In fact, a deep theorem due independently to Casson-Jungreis \cite{casson-jungreis}  and Gabai \cite{gabai-cgnce} says that if $\pG$ is $S^1$, then a finite index subgroup $G'$ of $G$ is a surface group, so that $G'$ admits a planar Cayley graph. }
This leaves open the following
case in the light of  Theorem~\ref{thm-kk}.

{For Questions~\ref{qn-kk} and \ref{qn-other} below, we assume implicitly that we have already passed to the full measure subset
	given by Theorem~\ref{thm-12proper}.}

\begin{qn}\label{qn-card}
	Is the cardinality of exceptional directions independent of  $\omega$ in the full measure subset
	given by Theorem~\ref{thm-12proper}?
\end{qn}

\begin{qn}\label{qn-kk}
	Let $G$ be a one-ended hyperbolic group with $\pG$ either  a Sierpinski carpet or  a Menger curve. What is the cardinality of exceptional directions?
\end{qn}

We also enunciate the following question (probably easier) that would complete the story.
\begin{qn}\label{qn-other}
	Let $G$ be 
	\begin{enumerate}
		\item either a virtually free group with some finite generating set,
		\item or a surface group with a \emph{non-planar} Cayley graph.
	\end{enumerate}
	What is the cardinality of exceptional directions? For a  virtually free group, when are there no exceptional directions?
	
	Let $G$ be free, and $\Gamma$ a Cayley graph with  non-standard generators. What is the cardinality of exceptional directions? 
\end{qn}

{As promised after the statement of Theorem~\ref{thm-nonfreedense}, we briefly discuss the issues involved in the case that $G$ is a virtually free group. For virtually free groups that are free products of finite groups, equipped with   generating sets given by generating sets of the finite group factors, it is easy to see that there are no exceptional directions. For instance,
	 let $G=\Z/2 * \Z/3$, and let $a, b$ generate $\Z/2, \Z/3$ respectively. Then for each bigon or triangle contained in the Cayley graph, random geodesics between any pair of points are unique almost surely. We now pass to the full measure subset $\Omega'$ of configurations for which each of 
	 the countably many  bigons or triangles satisfies the condition that 
	 random geodesics between any pair of vertices on any of them are unique.
	  Next, observe that any geodesic from $1$ to $\pG$ passes through a deterministic sequence of points $1=x_0, x_1, \cdots, x_n, \cdots$, where successive points are vertices of either the same bigon or of the same triangle. Since the  $\Omega'$-random geodesics between successive vertices $x_i, x_{i+1}$ are unique,  it follows that there are no exceptional directions almost surely.}
	 
	{  The test case $G=F_2 \times \Z/2$ with standard generating set needs a different kind of handling and is the first example not addressed by our techniques in this paper. The only way an exceptional direction may exist is if there are two parallel deterministic geodesic rays in a common direction (in $\pG$) that are also random geodesics. We need to estimate first the probability of some parallel pair of geodesic segments surviving in the ball of radius $n$. Next, one demands that for some such $n$, a parallel pair survives as $n$ goes to $\infty$. It is possible that carrying out such estimates is not too difficult; however, the above approach is far-removed from the techniques of this paper.}

\subsection{Poincare duality groups and exceptional directions}\label{sec-pd} The aim of this subsection is to strengthen Theorem~\ref{thm-omni-intro} for hyperbolic Poincare duality groups. In \cite{HN01} (see Remark following \cite[Theorem 1.10]{HN01}), the authors assert that in the Howard-Newman model with continuous rotational symmetry, exceptional directions in
$\R^{n+1}$ must have Hausdorff
dimension at least $n- 1$. We were motivated, in part, by this assertion
and provide below an analogous statement for hyperbolic Poincare duality groups (see \cite{bes-mess} and \cite[Section 1.4]{mahan-pattern} for relevant background). The following is the main theorem of this subsection.

\begin{theorem}\label{thm-pd}
{Let $G$ be a hyperbolic $PD(n+1)$ group. Then, for a.e.\ $\omega$, the set of exceptional directions in the forward tree $F(1,\omega)$ has topological dimension at least $(n-1)$. Hence the Hausdorff dimension of the set of exceptional directions  is at least  $(n-1)$ almost surely. If, in addition, $G$ is the fundamental group of a closed
negatively curved manifold, so that $\pG = S^n$, then the set of exceptional directions in $F(1,\omega)$ has topological dimension exactly $(n-1)$ almost surely.}
\end{theorem}

We start with a topological lemma.
\begin{lemma}\label{lem-sep}
Let $M$ be a closed orientable $n-$manifold,  $A, B\subset M$ be closed sets, both with non-empty interior such that $A\cup B=M$. Let $C = A \cap B$. Further, assume that there exist interior points $a, b$ in the complements of $B, A$ respectively.
Then $C$ has topological dimension {at least} $n- 1$.
\end{lemma}

\begin{proof} 
Since $M$ is a  manifold, $H_n (M, M-a, \mathbb{Z})= \mathbb{Z}= H_n (M, M-b,  \mathbb{Z})$.
Further, since $a,b$ are interior points of the closed sets $A, B$ respectively,   $H_n (A, A-a, {\mathbb{Z}})= {\mathbb{Z}}= H_n (B, B-b,  {\mathbb{Z}})$.
So,  ${\mathbb{Z}}\subset H_n (A, C,  {\mathbb{Z}})$ and ${\mathbb{Z}}\subset  H_n (B,C,{\mathbb{Z}})$.

We argue by contradiction. Suppose that
 $C$ has topological dimension at  most $(n-2)$.
 We now use the fact \cite{walsh} that if the topological  dimension of a set is finite, then the topological dimension equals the Cech (co)homological dimension over $\Z$, so that $H_{n-1} ( C, {\mathbb{Z}})=0$.
  Then the homology exact sequence of $(A, C)$ forces
$H_n (A, C,  {\mathbb{Z}}) =H_n(A, \mathbb{Z})$. 
Similarly, for the pair $(B, C)$.

Now, applying the Mayer-Vietoris for $M, A, B$, we find that since $C=A\cap B$  has topological dimension at  most  $(n-2)$,
$H_n(M) = H_n (A) \oplus H_n (B)$ and hence $\mathbb{Z} \oplus \mathbb{Z} < H_n(M)$, a contradiction.
\end{proof}

Note that the above argument used only the algebraic facts that 
 $H_n (M, M-a, \mathbb{Z})= \mathbb{Z}= H_n (M, M-b,  \mathbb{Z})$, and
 $H_n(M)=\mathbb{Z}$. These are exactly the defining conditions of a homology manifold.
 Recall that $M$ is said to be an \emph{$n-$dimensional homology manifold} if 
 \begin{enumerate}
 \item $M$ is a locally compact Hausdorff space  with finite homological dimension over $\Z$. 
 \item   For all $m \in M$, $H_n(M,M \setminus \{m\}, \mathbb{Z}) = \Z$ and $H_i(M,M \setminus \{m\}, \mathbb{Z}) =   0$ for $i \neq n$.
 \item $H_n(M, \mathbb{Z}) = \Z$.
 \end{enumerate}
  Thus, the proof of Lemma~\ref{lem-sep} goes through for $M$ a homology manifold verbatim. We record this:

 \begin{cor}\label{cor-sep}
 	Let $M$ be a homology manifold,  and $A, B\subset M$ be closed sets, both with non-empty interior such that $A\cup B=M$. Let $C = A \cap B$. Further, assume that there exist interior points $a, b$ in the complements of $B, A$ respectively.
 	Then $C$ has topological dimension {at least} $n- 1$.
 \end{cor}

We now recall the following theorem of Bestvina-Mess (see \cite[Section 1.4]{mahan-pattern} for a quick summary of the background material).

\begin{theorem}\cite{bes-mess}
\label{t:bes-messbd}
Let $G$ be a hyperbolic $PD(n+1)$ group. Then $\partial G$
is a homology manifold {of dimension $n$.}
\end{theorem}

\begin{proof}[Proof of Theorem \ref{thm-pd}]
Let $G$ is a hyperbolic $PD(n+1)$ group and $\Gamma$ a Cayley graph with respect to a finite symmetric generating set. Recall the event $\mathcal{C}_{\varepsilon, r}$ from Lemma \ref{l:wandering0}. By Lemma \ref{l:wandering0} and the Borel-Cantelli Lemma, $\liminf \mathcal{C}_{\varepsilon, r}$ is an almost sure event. For $\omega\in \liminf \mathcal{C}_{\varepsilon, r}$, let $r_0(\omega)$ be such that $\mathcal{C}_{\varepsilon, r}$ holds for all $r\ge r_0$. {Now, for $r\ge r_0$ and $v\in \partial B_r(1)$, define the \emph{ $\omega-$shadow} $\mho (v,\omega)$ of $v$ to be the collection of $\xi \in \pG$ such that there exists some $[1,\xi)_{\omega}$ satisfying the property that $v$ is the first point of intersection of 
	$[1,\xi)_{\omega}$ with $\partial B_r(1)$. Observe that
$\mho (v,\omega)$ is a closed subset of $\pG$. Indeed, suppose $\xi_n\in \mho (v,\omega)$ such that $\xi_n\to \xi$. Since the restriction of $[1,\xi_n)_{\omega}$ from $1$ to $v$ is contained in $B_r(1)$ for each $\xi_n$ it follows by a diagonal argument that there exist an infinite path $\{1=u_0,u_1,\ldots,\}$ from $1$ that is a  subsequential limit of $[1,\xi_n)_{\omega}$ that intersects $\partial B_r(1)$ first at $v$. By definition, each initial segment of this  path is an $\omega$-geodesic and hence this path itself is also an $\omega$-geodesic ray. Since $\xi_n\to \xi$ it follows that $u_{i}\to \xi$ and therefore there exists a $[1,\xi)_{\omega}$ which intersects $\partial B_r(1)$ first at $v$. Therefore $\xi\in \mho (v,\omega)$.}


{On the event $C_{\varepsilon,r}$ it follows that if $\xi\in \mho (v,\omega)$, then $v\in B_{\varepsilon r}(w)$ where $w=[1,\xi)\cap \partial B_r(1)$. Therefore $\mho (v,\omega)$ is contained in the deterministic $\varepsilon r$-shadow $\mathcal{S}(v,\varepsilon r)$ of $v$. Notice that by restricting $\omega$ to a further almost sure subset if necessary, we can assume by Theorem \ref{prop-ftree}  that $\pG$ is a finite union of closed sets $\cup_{u\in \partial B_r(1)} \mho (v,\omega)$.} Therefore, there exists $v\in \partial B_r(1)$ such that $\mho (v,\omega)$ has a non-empty interior (by the Baire category theorem, for instance).  Also, notice that, by choosing $\varepsilon$ sufficiently small there exists $w\in \partial B_r(1)$ such that the deterministic shadows $\mathcal{S}(v,\varepsilon r)$ and $\mathcal{S}(w,\varepsilon r)$ are disjoint. Indeed, given $k \in \natls$, if the Gromov inner product $\langle v, w\rangle_1 \leq k$, then the deterministic shadows $\mathcal{S}(v,D')$ and $\mathcal{S}(w,D')$ are disjoint for $v, w \in \partial B_{r'} (1)$
for $r' \geq 2k+D+\delta$. {Set $D=\varepsilon r$, and $r'=r$. Then for $\varepsilon \in (0,1)$, one can  
choose $r$  large enough depending on $\varepsilon$ such that $r' \geq 2k+D+\delta$. Therefore, for large enough $r$, there exists $v,w\in \partial B_r(1)$ such that  $\mathcal{S}(v,\varepsilon r)$ and $\mathcal{S}(w,\varepsilon r)$ are disjoint, as claimed}.  {Now, set $A=\mho (v,\omega)$, and let $B=\cup_{u \neq v, u \in \partial B_n(1)}\mho(u,\omega)$. Notice that $A$ and $B$ are closed subsets of $\pG$ with non-empty interiors ($A$ has non-empty interior by choice, and $B$ has non-empty interior since it contains $\mathcal{S}(w,\varepsilon r)$).} Also $A \cup B =\pG$.

	Then, by Corollary~\ref{cor-sep} and Theorem \ref{t:bes-messbd}, $ A\cap B$ has topological dimension at least $n-1$. {Observe now that $a\in A\cap B$ implies $a$ is an exceptional direction in $F(1,\omega)$. Indeed, any $a\in A\cap B$ belongs to  
$\mho(u,\omega)\cap \mho(v,\omega)$ for some $u$. Therefore there exist $[1,a)_{\omega}^1$ and $[1,a)^2_{\omega}$  whose first intersection points with $\partial B_r(1)$ are distinct. Thus, $a$ is an exceptional direction of $F(1,\omega)$.
	Therefore, the set of exceptional directions has topological dimension at least $n-1$.}
 
 To conclude the proof of Theorem~\ref{thm-pd}, it suffices to show that 
 the topological dimension of exceptional points cannot equal $n$ if
 $\pG =S^n$. This follows from \cite[Theorem IV 3]{hw-book} which asserts that
 if the topological dimension of a subspace of $\R^n$ is $n$, then it has non-empty interior.  The same holds for $S^n$ as the latter is the one-point compactification of $\R^n$.
 This is not possible by Theorem~\ref{t:exceptional} and Ahlfors regularity of $\pG$ which implies that every open set in $\pG$ must have positive Patterson-Sullivan measure.
\end{proof}

 We conclude this subsection with the following general question (without extra hypotheses on $G$).
 
 \begin{qn}
 If $G$ is hyperbolic with $\pG$ of topological dimension $n$, do the set of exceptional directions have topological dimension $(n-1)$?
 \end{qn}

\subsection{Bounds on multiplicity}\label{sec-bb} Theorems~\ref{thm-mult} and \ref{t:finite} furnish respectively upper and lower bounds on the  maximum multiplicity of exceptional directions. The first is in terms of the topological dimension of $\pG$, and the latter is, in some sense, related to the Hausdorff dimension of $\pG$. 
It is not clear where exactly in this range the maximum multiplicity lies.
In particular, we do not know if the answer  depends on $\rho$, the passage time distribution on edges. As mentioned in the Introduction, the case where a complete answer is known is  exponential planar LPP, where the N3G problem is fully answered. In the hyperbolic setup of the present paper, to push our approach further, one needs effective coalescence estimates between hyperplanes to improve the upper bound of Theorem~\ref{t:finite}.

We turn to the lower bound of Theorem~\ref{thm-mult} now.
The maximum multiplicity of exceptional directions for the classical (deterministic) Cannon-Thurston map \cite{CTpub}  $i_{CT}: S^1 \to S^2$  is  3 for \emph{generic degenerate}
Kleinian surface groups (see \cite{mahan-icm} for a survey of results).  It is not hard to pre-compose the above $i_{CT}$ with a surjective map $j: \partial T_n \to S^1$ from the boundary Cantor set of an $n-$regular tree $T_n$ onto $S^1$ such that $j \circ i_{CT}: \partial T_n \to S^2$ also has maximum multiplicity 3.
 The number  3  comes from the fact that the pre-images of exceptional directions in $S^1$ under the Cannon-Thurston map correspond to end-points of leaves of laminations
 \cite{mahan-elct} (see \cite{CTpub} for a proof in the special case that the surface group corresponds to the fiber of a hyperbolic 3-manifold fibering over the circle). Further, a generic lamination on a surface only has ideal triangles in its complement.  But there are indeed (non-generic) degenerate
Kleinian surface groups  for which the multiplicity is much larger than this topological lower bound of 3. Hence, even in the deterministic setup, multiplicity can indeed be large. From this point of view, it is not reasonable to expect that, without further assumptions, the topological lower bound on maximum multiplicity established in Theorem~\ref{thm-mult}
actually equals the maximum multiplicity of exceptional directions. We thus ask the following.

\begin{qn}\label{qn-mult} $ $
	\begin{enumerate}
	\item If $\partial G_1 = \partial G_2$, and the passage time distribution $\rho$ is the same for Cayley graphs $\Gamma_1, \Gamma_2$ (of $G_1, G_2$ respectively), is the maximum multiplicity the same for $\Gamma_1, \Gamma_2$? In particular, is this true if  $\Gamma_1, \Gamma_2$ are quasi-isometric? If $G_1=G_2$, and 
	 $\Gamma_1, \Gamma_2$ are different Cayley graphs of the same group, is the maximum multiplicity the same for $\Gamma_1, \Gamma_2$?
	 \item Is there a probability distribution $\rho$ satisfying Assumption~\ref{assume-subexp} such that the lower bound in Theorem~\ref{thm-mult} is strictly less than the maximum multiplicity?
	\end{enumerate}
\end{qn}

\subsection{Relationships between forward and backward trees}\label{sec-treeqns}
As promised at the beginning of Section~\ref{sec-direx}, we briefly recall the essential features of the KAN decomposition.\\

\noindent {\bf  $PSl(2,\R), \Hyp^2$ and  random forward, backward trees:}\\
We provide a conceptual reason as to why  forward and backward trees have similar qualitative asymptotic properties as illustrated in the main theorems of this paper. Let $L=PSl(2,\R)= Sl(2,\R)/ (\pm I)$. Then the Iwasawa  or $KAN$ decomposition of $L$ is given as follows.
\begin{enumerate}
	\item $K=PSO(2)=SO(2)/ (\pm I) $ is the maximal compact in $L$,
	\item $A$ is the space of diagonal matrices with entries $\lambda, 1/\lambda$, $\lambda > 1$,
	\item $N$ is the space of strictly upper triangular $2 \times 2$ matrices with $1$'s along the diagonal.
\end{enumerate}

Then the associated symmetric space $L/K$ equals  $\Hyp^2$, and there are two possible identifications:
\begin{enumerate}
	\item $\Hyp^2 = KA$, where $A$ gives the radial coordinate, and $K$ gives the angular coordinate.
	\item $\Hyp^2 = AN$, where $A$ is identified with the vertical $y-$axis in the upper half-plane model, and $N$ consists of horizontal translations.
\end{enumerate}

The first description $\Hyp^2 = KA$ has a distinguished origin $o$ about which $K$ rotates, and $A$ gives the radial distance from $o$. The  forward tree may be regarded as the FPP analog of this model. The second 
description $\Hyp^2 = AN$ has a distinguished $\infty$ kept fixed by both $N$  and $A$. The   backward tree may be regarded as the FPP analog of this model where $\infty$ is replaced by $\xi$. 

Let $o$ be a base-point in $\Hyp^2$ as above,  $z \in \partial \Hyp^2$, and $[o,z)$ be a geodesic ray. Let $z_n \in [o,z)$ be such that
$d(z_n, o) =n$. Then the balls $B_n(z_n)$ of radius $n$ about $z_n$ converge to the unique horodisk $\HHH(z)$ based at $z $ and containing $o$ as  a point on 
its boundary horocycle  $\partial \HHH(z)$. This geometric construction generalizes to hyperbolic groups in a straightforward manner by replacing $\Hyp^2$ with a Cayley graph $\Gamma$, the base-point $o$ by $1$ and  $z \in \partial \Hyp^2$ by  $z \in \pG$. 
More generally, this construction  shows that horocycles (i.e.\ orbits of $N$ in $\Hyp^2$) are limits of boundaries $\partial D(n)$ of a suitably chosen sequence of $n-$disks where
$\partial D(n)$ passes through a fixed point (the origin, say) in  $\Hyp^2$. It is this fact that allows us to move from the first to the second descriptions of $\Hyp^2$ to deduce ergodic theoretic results
(see \cite{eskin-mcmullen} and subsequent works for instance). In our present random discrete setup, it shows up in similar asymptotic behavior of  forward and backward trees. There are, however, subtle differences that this analogy does not capture.  Question~\ref{qn-fwdbwd}  below captures one such feature.
\\

For a backward tree $T(\xi,\omega)$, the base-point is chosen arbitrarily; hence the restriction of the Cannon-Thurston map $i_\omega$ 
to $\partial T(\xi,\omega)$ does not depend on the choice of base-point. The same is not true for the forward tree, as the precise relationship between the trees
$F(1,\omega)$ and $F(g, \omega)$ is not clear. We thus ask the following.

\begin{qn}\label{qn-fwdbwd}
How do the exceptional directions and their multiplicities of the forward tree $F(g, \omega)$ depend on base-point $g$? In particular, are they independent of $g$?
The test case is when the support of $\rho$ is bounded away from $0, \infty$.
\end{qn}
In  planar (Euclidean) exactly solvable LPP, it was shown in \cite[Theorem 3.9]{jrs} that exceptional directions are independent of the  base-point $g$. In conjunction with the resolution
of the N3G problem \cite{coupier11} for planar (Euclidean) exactly solvable LPP, the second part of Question~\ref{qn-fwdbwd} above automatically has a positive answer in this case.
In the deterministic case, where $ \rho$ has an atom at 1, Question~\ref{qn-fwdbwd} specializes to the full geodesic language
(without an ordering on the generating set) generated by the Cannon automaton.

An approach to resolving Question~\ref{qn-fwdbwd} is to use the fact, already stated, that for the backward tree $T(\xi,\omega)$, exceptional directions and 
 their multiplicities are tautologically independent of the  base-point $g$. Recall the  well-known geometric limiting construction of disks to horodisks describes earlier in this subsection.
 We then ask the following. 
 
 \begin{qn}\label{qn-cgnce}
 Do the forward trees $F(z_n, \omega)$ converge to $T(z, \omega)$ in a suitable sense?  
 \end{qn}

The strongest sense in which such a convergence could occur is Gromov-Hausdorff. We do not expect a positive answer in such a strong sense. However, one might expect a positive answer if one  wants only a convergence in measure. This is due to the following fact: if $z_n$ is chosen uniformly randomly from 
the ball $B_n(1)$ of radius $n$, then the uniform probability on $B_n(1)$
converges to the Patterson-Sullivan measure on $\pG$.
Hence, an appropriate modification  of Benjamini-Schramm  convergence might yield a positive answer to Question~\ref{qn-cgnce}.

\subsection{Generalizations}\label{sec-genlzns}
To conclude this paper, we indicate briefly how, in many of our results, the Cayley graph $\Gamma$ of a hyperbolic group may be replaced with a $\delta-$hyperbolic graph $X$, with some additional assumptions. Assume therefore that 
 \begin{assume}\label{assume-Xhyp}
  $X$ is a $\delta-$hyperbolic graph with a uniform upper bound $D$ on the valence of any vertex.
 \end{assume}

 \noindent 1) The results from \cite{BM19} used in this paper, i.e.\ all results from Sections 6-8 of that paper, \cite[Lemma 5.14]{BM19},
 \cite[Lemma 4.3]{BM19}, and
  \cite[Theorem 2.14]{BM19} hold under the general Assumption~\ref{assume-Xhyp}.\\

  \noindent 2) Observe that Corollary~\ref{cor-effbt} is stated in this generality.
 Hence, Theorems~\ref{prop-effbt}, \ref{thm-bi}, \ref{thm-btree} go through for $X$ as in Assumption~\ref{assume-Xhyp}. Thus, forward and backward trees in $X$ have complete ends under Assumption~\ref{assume-subexp}.\\

  \noindent 3a) The existence of surjective Cannon-Thurston maps for forward and backward trees  in Theorem~\ref{thm-12proper} uses Theorems~\ref{prop-effbt}, \ref{thm-bi}, \ref{thm-btree}. The other arguments in its proof being completely general, Theorem~\ref{thm-12proper} continues to hold under Assumption~\ref{assume-subexp}.\\

   \noindent 3b)  Stallings' Theorem, Dunwoody's accessibility Theorem~\ref{dunwoody} and Bowditch-Swarup's Theorem~\ref{thm-lc} have no analogs for a general $X$. However, Lemma~\ref{lem-top} is fully general.
   Hence, for Theorem~\ref{thm-nonfreedense} to go through for $X$, it suffices to assume the following (since the other ingredient, existence of surjective Cannon-Thurston maps, goes through under Assumption~\ref{assume-subexp}).
   
   \begin{assume}\label{assume-bdy}
   	For $X$ satisfying Assumption~\ref{assume-Xhyp}, assume further that 
   $\partial X$ is connected and locally connected.
   \end{assume}
   
      \noindent 3c) Theorem~\ref{thm-uncountable} uses only basic properties of topological dimension. Hence, it goes through for $X$ satisfying Assumption~\ref{assume-Xhyp}, and \ref{assume-bdy}.\\
  
   \noindent 3d)  For the same reasons as (3c) above,  Theorem~\ref{thm-mult} goes through for $X$ satisfying Assumption~\ref{assume-Xhyp}, and \ref{assume-bdy}.\\
 
  \noindent 4a) In Section~\ref{sec-exceptional-coalesce}, the proof of Theorem~\ref{t:exceptional} uses the Patterson-Sullivan measure on $\pG$. The other ingredients in the proof of Theorem~\ref{t:exceptional}, i.e.\ Lemmas~\ref{l:exceptionallift}, \ref{l:wandering0} and \ref{l:wandering} go through under the general Assumption~\ref{assume-Xhyp}. To ensure that we can use the argument involving the  Patterson-Sullivan measure, it thus suffices to assume the following, which extracts from \cite{coornert-pjm} the output of the relevant group-theoretic argument, and treats it as an axiom
  (see \cite[2.5.8,2.5.9,2.5.10]{calegari-notes} for instance).
 
  \begin{assume}\label{assume-ps}
 For $X$ satisfying Assumption~\ref{assume-Xhyp}, assume further that 
 \begin{enumerate}
 \item there exists $C, \lambda \geq 1$ such that for any $x \in X$ and $n \in \natls$, $\frac{1}{C} \lambda^n \leq |B_n(x)| \leq {C} \lambda^n.$
 \item  the limit of uniform measures on finite balls converge to an Ahlfors-regular measure of Hausdorff dimension $\lambda$ on the boundary.
 \end{enumerate}
 
 \end{assume}
Assumption~\ref{assume-ps} is satisfied for a larger class of spaces than Cayley graphs of hyperbolic groups. In particular, it is satisfied whenever a graph admits a (not necessarily discrete) cocompact group of isometries. It is also satisfied for Cayley graphs with a thin set of vertices removed.
 
 Thus Theorem~\ref{t:exceptional} goes through for  $X$ satisfying Assumptions~\ref{assume-Xhyp} and \ref{assume-ps}.\\

 \noindent 4b) The hyperplane construction and the estimates of Lemma~\ref{l:finite}
 go through for  $X$ satisfying Assumptions~\ref{assume-Xhyp}. Hence, 
 Theorem~\ref{t:finite} follows for  $X$ satisfying Assumptions~\ref{assume-Xhyp}.

\bibliography{fpphoro}
\bibliographystyle{alpha}

\end{document}